\documentclass[12pt,a4paper]{amsart}

\usepackage[utf8]{inputenc}
\usepackage{amssymb}
\usepackage{mathtools}
\usepackage[mathscr]{euscript}
\usepackage{mathrsfs}
\usepackage{stmaryrd}
\usepackage[normalem]{ulem}
\usepackage{enumitem}
\setenumerate[1]{label=(\alph*)}
\setenumerate[2]{label=(\roman*)}
\usepackage[all]{xy}
\usepackage{fullpage}

\usepackage{xr-hyper}
\usepackage{hyperref}

\hypersetup{  
  pdftitle    = {Unbounded derived categories of small and big modules: Is
    the natural functor fully faithful?},
  pdfauthor  = {Leonid Positselski, Olaf M.~Schn{\"u}rer}, 
  colorlinks = true}

\newtheorem{theorem}{Theorem}[section]
\newtheorem*{theorem*}{Theorem}
\newtheorem{proposition}[theorem]{Proposition}
\newtheorem{lemma}[theorem]{Lemma}
\newtheorem{corollary}[theorem]{Corollary}

\theoremstyle{definition}

\newtheorem{example}[theorem]{Example}

\theoremstyle{remark}
\newtheorem{remark}[theorem]{Remark}

\DeclareMathOperator{\image}{im}

\DeclareMathOperator{\Hom}{Hom}

\DeclareMathOperator{\Ext}{Ext}
\DeclareMathOperator{\Tot}{Tot}
\DeclareMathOperator{\Ac}{Ac}

\DeclareMathOperator{\Cone}{Cone}
\DeclareMathOperator{\colim}{colim}

\DeclareMathOperator{\B}{B}
\DeclareMathOperator{\Z}{Z}
\DeclareMathOperator{\HH}{H}

\DeclareMathOperator{\Spec}{Spec}
\DeclareMathOperator{\Proj}{Proj}
\DeclareMathOperator{\proj}{proj}
\DeclareMathOperator{\inj}{inj}

\newcommand{\pr}{{\operatorname{pr}}}

\newcommand{\mfp}{{\mathfrak{p}}}

\renewcommand{\epsilon}{{\varepsilon}}

\renewcommand{\phi}{{\varphi}}

\newcommand{\define}[1]{{\textbf{#1}}}

\newcommand{\id}{{\operatorname{id}}}

\newcommand{\bZ}{{\mathbb{Z}}}

\newcommand{\bC}{{\mathbb{C}}}
\newcommand{\bN}{{\mathbb{N}}}

\DeclareMathOperator{\rad}{rad}

\DeclareMathOperator{\Qcoh}{Qcoh}
\DeclareMathOperator{\coh}{coh}

\DeclareMathOperator{\Mod}{Mod}
\let\mod\relax
\DeclareMathOperator{\mod}{mod}

\newcommand{\D}{{\operatorname{D}}}
\newcommand{\K}{{\operatorname{K}}}

\newcommand{\ol}[1]{{\overline{#1}}}

\newcommand{\abs}{{\mathrm{abs}}}

\newcommand{\ra}{\rightarrow}

\newcommand{\xra}[2][]{\xrightarrow[{#1}]{#2}}
\newcommand{\xla}[2][]{\xleftarrow[{#1}]{#2}}
\newcommand{\sira}{\xra{\sim}}
\newcommand{\xsira}[1]{\xrightarrow[\sim]{#1}}

\newcommand{\sra}{\twoheadrightarrow}
\newcommand{\hra}{\hookrightarrow}

\makeatletter
\providecommand*{\twoheadrightarrowfill@}{%
  \arrowfill@\relbar\relbar\twoheadrightarrow
}
\providecommand*{\twoheadleftarrowfill@}{%
  \arrowfill@\twoheadleftarrow\relbar\relbar
}
\providecommand*{\xsra}[2][]{%
  \ext@arrow 0579\twoheadrightarrowfill@{#1}{#2}%
}
\providecommand*{\xsla}[2][]{%
  \ext@arrow 5097\twoheadleftarrowfill@{#1}{#2}%
}
\providecommand*{\hookrightarrowfill@}{%
  \arrowfill@\relbar\relbar\hookrightarrow
}
\providecommand*{\hookleftarrowfill@}{%
  \arrowfill@\hookleftarrow\relbar\relbar
}
\providecommand*{\xhra}[2][]{%
  \ext@arrow 0579\hookrightarrowfill@{#1}{#2}%
}
\providecommand*{\xhla}[2][]{%
  \ext@arrow 5097\hookleftarrowfill@{#1}{#2}%
}
\makeatother

\newcommand{\sptag}[1]{\href{http://stacks.math.columbia.edu/tag/#1}{#1}}

\makeatletter
\@namedef{subjclassname@2020}{%
  \textup{2020} Mathematics Subject Classification}
\makeatother

\numberwithin{equation}{section}

\title{Unbounded derived categories of small and big modules: Is
  the natural functor fully faithful?}

\author{Leonid Positselski}

\address{Institute of Mathematics of the Czech Academy of Sciences\\
  \v Zitn\'a~25, 115~67 Prague~1\\
  Czech Republic; and
  \newline\indent
  Laboratory of Algebra and Number Theory\\
  Institute for Information Transmission Problems\\
  Moscow 127051\\
  Russia} 

\email{positselski@math.cas.cz} 

\author{Olaf M.~Schn{\"u}rer}

\address{Institut f\"ur Mathematik\\ 
  Universit{\"a}t Paderborn\\
  Warburger Stra\ss{}e 100\\
  33098 Pa\-der\-born\\
  Germany
}

\email{olaf.schnuerer@math.uni-paderborn.de}

\begin{document}


\begin{abstract}
  Consider the obvious functor
  from the unbounded derived category of all finitely generated
  modules over a
  left
  noetherian ring $R$ to the unbounded derived
  category of all modules.  We answer the natural question
  whether this functor defines an equivalence onto the full
  subcategory of complexes with finitely generated cohomology
  modules in two special cases.  If $R$ is a quasi-Frobenius
  ring
  of infinite global dimension,
  then
  this functor is not full.  If $R$ has finite
  left
  global dimension,
  this functor is an equivalence.  We also prove variants of the
  latter assertion for left coherent rings, for noetherian
  schemes and for locally noetherian Grothendieck categories.
\end{abstract}

\maketitle


\section{Introduction}
\label{sec:introduction}

Let $R$ be a left noetherian ring. Then the inclusion
$\mod(R) \subset \Mod(R)$ of the category of all finitely
generated
left
$R$-modules into the category of
all
left
$R$-modules
is trivially fully faithful and exact. We ask when the induced
functor
\begin{equation*}
  \D(\mod(R)) \ra \D(\Mod(R))  
\end{equation*}
on unbounded derived categories is fully faithful as well and
when
its essential image is the subcategory
$\D_{\mod(R)}(\Mod(R)) \subset \D(\Mod(R))$ of complexes with
cohomology objects in $\mod(R)$. Put together, we study the
question when the induced functor
\begin{equation*}
  F \colon \D(\mod(R)) \ra \D_{\mod(R)}(\Mod(R))  
\end{equation*}
is an equivalence. We present a negative and two positive results.
Before explaining these results, we
would like to mention that the answer to the corresponding
question in the bounded above case is well-known: The induced
functor on bounded above derived categories is always an
equivalence $\D^-(\mod(R)) \sira \D^-_{\mod(R)}(\Mod(R))$
(special case of Proposition~\ref{p:D-minus-equivalence}).

The negative result is that $F$ is neither full nor essentially
surjective
(not
even up to direct summands) if $R$ is a
quasi-Frobenius ring (= a left noetherian, left self-injective
ring) of infinite global dimension
(Theorem~\ref{t:self-injective-artin-gdim-infinite}).
Quasi-Frobenius 
rings are always left and right noetherian and left and right
artinian.  The ring $k[\epsilon]/(\epsilon^2)$ of dual numbers
over a field $k$ and $\bZ/4\bZ$ are examples of commutative
quasi-Frobenius rings of infinite global dimension.

The first positive result is that $F$ is an equivalence whenever
$R$ has finite left global dimension
(Corollary~\ref{c:equiv-Dmod-DMod-gdim-finite}). Commutative
examples of such rings are polynomial rings $k[x_1, \dots, x_n]$
over a field $k$ and all regular quotients of such rings.  In
fact, Corollary~\ref{c:equiv-Dmod-DMod-gdim-finite} is the
slightly stronger statement that $F$ is an equivalence whenever
$R$ is a left coherent ring of finite left global dimension and
$\mod(R)$ denotes the category of all finitely presented
left
$R$-modules.

Of course there are many examples of left noetherian rings which
have infinite left global dimension and are not
left
self-injective,
even if one restricts attention to finite-dimensional algebras
over a field or to commutative noetherian rings.  It would be
nice to characterize, in terms of classical homological algebra,
all left noetherian (or coherent) rings $R$ for which $F$ is an
equivalence.

The second positive result is a generalization of the first
positive result to algebraic geometry.  If $X$ is a regular
noetherian scheme of finite Krull dimension, then the functor
induced by the inclusion $\coh(X) \subset \Qcoh(X)$ is an
equivalence
\begin{equation*}
  \D(\coh(X)) \sira \D_{\coh}(\Qcoh(X))
\end{equation*}
(Corollary~\ref{c:regular-noetherian-finite-dim}).

Similarly as above, it would be nice to characterize all
noetherian schemes $X$ for which the obvious functor
\begin{equation*}
  E \colon \D(\coh(X)) \ra \D_{\coh}(\Qcoh(X))
\end{equation*}
is an equivalence.  Again, the bounded above case is a classical
result: The functor $\D^-(\coh(X)) \sira \D^-_{\coh}(\Qcoh(X))$
is always an equivalence
(\cite[Expos\'e~II, Proposition~2.2.2,
p.~167]{berthelot-grothendieck-illusie-SGA-6},
\cite[\sptag{0FDA}]{stacks-project}).  Our negative result shows
that there are many affine schemes $X$, for example
$\Spec k[\epsilon]/(\epsilon^2)$ and $\Spec \bZ/4\bZ$, such that
$E$ is not an equivalence.\bigskip

Our two positive results are in fact consequences of abstract
categorical results. Assume that $\mathcal{A}$ is an abelian
category and that $\mathcal{U} \subset \mathcal{A}$ is an abelian
subcategory closed under extensions.  Then it is natural to ask
when
\begin{equation*}
  G \colon \D(\mathcal{U}) \ra \D_\mathcal{U}(\mathcal{A})
\end{equation*}
is an equivalence. If $\mathcal{A}$ has finite global dimension
and some technical conditions are satisfied, the answer to this
question is affirmative
(Theorems~\ref{t:equiv-abelian-subcat-gdim-finite} and
\ref{t:serre-subcat-gdim-finite};
Corollary~\ref{c:equiv-hereditary} in the hereditary case).

An important special case where $G$ is an equivalence is the case
that $\mathcal{A}$ is a locally noetherian Grothendieck category
of finite global dimension
and that $\mathcal{U} \subset \mathcal{A}$ is its full
subcategory of noetherian objects
(Corollary~\ref{c:locally-noetherian-category}). Let us say a few
words about the proof of this assertion.  The argument consists
of three steps.

Firstly, for any locally coherent Grothendieck category
$\mathcal{A}$ and its full subcategory
$\mathcal{U}\subset \mathcal{A}$ of coherent objects, the functor
$G$ is essentially surjective provided that $\mathcal{A}$ has
finite global dimension (special case of
Proposition~\ref{p:essentially-surjective}; note that
$\mathcal{U}$ has finite global dimension by Proposition~\ref{p:D-minus-equivalence}).

Secondly, for any abelian category $\mathcal{A}$, consider its
so-called absolute derived category $\D^\abs({\mathcal{A}})$.
Then, for any abelian category $\mathcal{A}$ of finite global
dimension, the canonical functor
$\D^\abs({\mathcal{A}})\ra \D(\mathcal{A})$ is an equivalence
(Remark~\ref{r:gdim-finite-acyclic=absolute-acyclic}).

Thirdly, for any locally noetherian Grothendieck category
$\mathcal A$, the obvious triangulated functor
\begin{equation*} 
  \D^\abs(\mathcal{U})\ra\D^\abs(\mathcal{A})
\end{equation*}
is fully faithful (special case of
Theorem~\ref{t:absolute-derived-fully-faithful}). Notice that
this holds for categories $\mathcal{A}$ of infinite global
dimension as well.  So, in this particular respect, absolute
derived categories are better behaved than conventional unbounded
derived categories.

If now $\mathcal{A}$ is a locally noetherian Grothendieck
category of finite global dimension, then its subcategory
$\mathcal{U}$ of noetherian objects has finite global dimension
as well and we deduce from the above three steps and
commutativity of the diagram
\begin{equation*}
  \xymatrix{
    {\D^\abs(\mathcal{U})} \ar[r] \ar[d] & 
    {\D^\abs(\mathcal{A})} \ar[d] \\
    {\D(\mathcal{U})} \ar[r] & 
    {\D(\mathcal{A})}
  }
\end{equation*}
that $G$ is an equivalence.

\subsection*{Acknowledgment}
\label{sec:acknowledgment}

We thank the organizers of the summer school \textit{Two weeks of
  silting} in Stuttgart 2019 where we obtained our first negative
result for the dual numbers.  The second author thanks Ciprian
Modoi for useful comments.
We thank the referee for reading our text carefully and for spotting
several small inaccuracies.
We thank Pierre Schapira for providing references.
The first author is supported by the
GA\v CR project 20-13778S and research plan RVO: 67985840.

\section{Notation and Preliminaries}
\label{sec:notat-prel}

Let $R$ be a left coherent ring.  The abelian category of all
left $R$-modules is denoted by $\Mod(R)$ and its full abelian
subcategory of all finitely presented left $R$-modules by
$\mod(R)$.  We just say module instead of left module in the
following.  The fully faithful, exact inclusion functor
$\mod(R) \ra \Mod(R)$ between abelian categories induces a
triangulated functor
\begin{equation*}
  F \colon \D(\mod(R)) \ra \D(\Mod(R))
\end{equation*}
between the corresponding unbounded derived categories.  This
functor certainly lands in $\D_{\mod(R)}(\Mod(R))$, the full
subcategory of $\D(\Mod(R))$ consisting of all objects $C$ with
$\HH^i(C) \in \mod(R)$ for all $i \in \bZ$.  Since $F$ obviously
commutes with cohomology, it reflects the zero object and hence,
being a triangulated functor, is conservative (= reflects
isomorphisms).  We remind the reader that $\D(\Mod(R))$ has all
coproducts and products and that these can be formed naively
because coproducts and products of quasi-isomorphisms
between complexes in $\Mod(R)$
are
quasi-isomorphisms. Note also that $\D(\Mod(R))$ is
idempotent-complete, and hence so is $\D_{\mod(R)}(\Mod(R))$.

In the case of a left noetherian ring $R$, it is well-known that
the functor induced by $F$ is an equivalence
\begin{equation}
  \label{eq:D-minus-mod-Mod}
  \D^-(\mod(R)) \sira \D_{\mod(R)}^-(\Mod(R))  
\end{equation}
of triangulated categories (e.\,g.\
\cite[\sptag{0FCL}]{stacks-project}, \cite[Proof of
Proposition~3.5]{huy-fourmukai}); here $\D^-$ indicates the vanishing
of all sufficiently high cohomologies.  This assertion holds more
generally for any left coherent ring $R$; this is a consequence
of Proposition~\ref{p:D-minus-equivalence} below.

Let now $\mathcal{A}$ be an abelian category and let
$\mathcal{U}\subset\mathcal{A}$ be a non-empty, full subcategory
closed under kernels, cokernels, and extensions.  Then
$\mathcal{U}$ is an abelian category as well, and the fully
faithful, exact inclusion functor $\mathcal{U}\ra\mathcal{A}$
induces a triangulated functor
\begin{equation*}
  \D(\mathcal{U})\ra\D(\mathcal{A})
\end{equation*}
between the unbounded derived categories.  Once again, this
functor
is conservative and
lands in
$\D_{\mathcal{U}}(\mathcal{A}) \subset\D(\mathcal{A})$, the full
subcategory of all complexes in $\mathcal{A}$ with cohomology
objects belonging to $\mathcal{U}$.
One would like to know when
$\D(\mathcal{U}) \ra \D_{\mathcal{U}}(\mathcal{A})$
is an
equivalence of categories.  The following result is 
well-known (e.\,g.\
\cite[Proposition~1.7.11]{kashiwara-schapira-sheaves-on-manifolds}
\cite[Theorem~13.2.8]{kashiwara-schapira-categories-and-sheaves});
we include its proof for convenience.

\begin{proposition}
  \label{p:D-minus-equivalence}
  Let $\mathcal{A}$ be an abelian category and let
  $\mathcal{U} \subset \mathcal{A}$ be a non-empty, full
  subcategory closed under kernels, cokernels and extensions.
  Assume that for any epimorphism $A\sra U$ in $\mathcal{A}$ with
  $U\in\mathcal{U}$ there exists a morphism $V\ra A$ in
  $\mathcal{A}$ with $V\in\mathcal{U}$ such that the composition
  $V\ra A\sra U$ is an epimorphism.  Then the obvious functor is
  an equivalence
  $\D^-(\mathcal{U}) \sira \D^-_{\mathcal{U}}(\mathcal{A})$.
\end{proposition}

\begin{remark}
  If we replace the assumption that the full subcategory
  $\mathcal{U}$ in Proposition~\ref{p:D-minus-equivalence} is
  closed under extensions by the assumption that it is closed
  under finite coproducts, it will automatically be closed under
  extensions.  Indeed, let $S \hra E \sra Q$ be a short exact
  sequence in $\mathcal{A}$ with $S, Q \in \mathcal{U}$.  By
  assumption, there is an object $F \in \mathcal{U}$ and a
  morphism $u \colon F \ra E$ such that $F \xra{u} E \sra Q$ is
  an epimorphism. We obtain a morphism
    \begin{equation*}
      \xymatrix{
        {0} \ar[r] &
        {S} \ar[r] &
        {E} \ar[r] &
        {Q} \ar[r] &
        {0}\\
        {0} \ar[r] &
        {S \times_E F} \ar[r] \ar[u] &
        {F} \ar[r] \ar[u]^-{u}&
        {Q} \ar[r] \ar[u]^-{\id} &
        {0}
      }
    \end{equation*}
    of short exact sequences. Note that $S \times_E F$ is in
    $\mathcal{U}$ as the kernel of $F \ra Q$. The left square is
    cartesian but also cocartesian. Hence $E$ is the cokernel of
    the morphism $S \times_E F \ra S \oplus F$ between objects of
    $\mathcal{U}$ and hence in $\mathcal{U}$.
\end{remark}

\begin{proof}
  Let $A$ be a complex in $\mathcal{A}$ with cohomologies
  $\HH^i(A) \in \mathcal{U}$, for all $i \in \bZ$.  Let
  $N \in \bZ$ and assume that $A^i \in \mathcal{U}$ for all
  integers $i > N$.  The cocycle object $\Z^{N+1}(A)$ is in
  $\mathcal{U}$ as the kernel of $A^{N+1} \ra A^{N+2}$.  The
  coboundary object $\B^{N+1}(A)$ is then in $\mathcal{U}$ as the
  kernel of $\Z^{N+1}(A) \ra \HH^{N+1}(A)$.  Therefore, by
  assumption, there is an object $U \in \mathcal{U}$ and a
  morphism $u \colon U \ra A^N$ such that the composition
  $U \xra{u} A^N \sra \B^{N+1}(A)$ is an epimorphism.  Similarly,
  there is an object $V \in \mathcal{U}$ and a morphism
  $v \colon V \ra \Z^N(A)$ such that the composition
  $V \xra{v} \Z^N(A) \sra \HH^N(A)$ is an epimorphism.  The
  diagram
  \begin{equation*}
    \xymatrix{
      {\dots} \ar[r] &
      {A^{N-2}} \ar[r]^-{d} &
      {A^{N-1}} \ar[r]^-{d} &
      {A^N} \ar[r]^-{d} &
      {A^{N+1}} \ar[r] &
      {\dots}\\
      {\dots} \ar[r] &
      {A^{N-2}} \ar[u]^-{\id} \ar[r]^-{\langle d, 0\rangle} &
      {A^{N-1} \times_{A^N} (V \oplus U)} \ar[u]^-{\pr_1}
      \ar[r]^-{\pr_2} & 
      {V \oplus U} \ar[u]^-{(v,u)} \ar[r]^-{(0,d \circ u)} &
      {A^{N+1}} \ar[u]^-{\id} \ar[r] &
      {\dots}\\
    }
  \end{equation*}
  is clearly commutative and its lower row is a complex in
  $\mathcal{A}$.  The obvious diagram chases together with the
  Freyd--Mitchell embedding theorem show that the vertical arrows
  define a quasi-isomorphism between the two rows. Note that the
  component $V \oplus U$ of the lower complex is in
  $\mathcal{U}$.

  Given any object $A \in \D^{-}_\mathcal{U}(\mathcal{A})$ there
  is an integer $N$ such that $\tau_{\leq N}A \ra A$ is a
  quasi-isomorphism
  where $\tau_{\leq N}A$ is the intelligent
  truncation. If we iterate the above construction we find
  a complex $L$ in $\mathcal{U}$ together with a
  quasi-isomorphism $L \ra \tau_{\leq N} A$.  This shows that for
  any object $A \in \D^{-}_\mathcal{U}(\mathcal{A})$ there is an
  object $L\in \D^{-}(\mathcal{U})$ together with a
  quasi-isomorphism $L \ra A$.  This gives essential surjectivity
  of our functor and also fully faithfulness by using the
  description of morphisms by roofs.
\end{proof}

\begin{remark}
  Assume that the morphisms $U \ra A^N$ and $V \ra \Z^N(A)$ in
  the above proof are monomorphisms, i.\,e.\ $U \subset A^N$ and
  $V \subset \Z^N(A)$. Then one can replace $V \oplus U$ in the
  above diagram by $V+U \subset A^N$ and, proceeding in this way,
  eventually obtain a quasi-isomorphic subcomplex of $A$ with
  terms in $\mathcal{U}$.
\end{remark}

The full subcategory of projective resp.\ injective objects in an
abelian category $\mathcal{A}$ is denoted below by
$\proj({\mathcal{A}})$ resp.\ $\inj(\mathcal{A})$.  We abbreviate
$\Proj(R):=\proj(\Mod(R))$ and $\proj(R):= \proj(\mod(R))$
and $\inj(R):=\inj(\mod(R))$.
Note that $\proj(R)=\Proj(R) \cap \mod(R)$.

We write $\K(\mathcal{A})$ for the homotopy category of an
additive category $\mathcal{A}$.  Let $\mathcal{A}$ be an abelian
category.  Then an object $P\in\K(\mathcal{A})$ is called
\define{h-projective} if $\Hom_{\K(\mathcal{A})}(P,A)=0$ for all
acyclic complexes $A\in\K(\mathcal{A})$.  For any h-projective
complex $P$ and any complex $K\in\K(\mathcal{A})$, the map
$\Hom_{\K(\mathcal{A})}(P,K)\ra\Hom_{\D(\mathcal{A})}(P,K)$
induced by the functor $\K(\mathcal{A})\ra\D(\mathcal{A})$ is an
isomorphism.  Similarly, an object $I\in\K(\mathcal{A})$ is said
to be \define{h-in\-jec\-tive} if $\Hom_{\K(\mathcal{A})}(A,I)=0$
for 
all acyclic complexes $A\in\K(\mathcal{A})$.  For any h-injective
complex $I$ and any complex $K\in\K(\mathcal{A})$, the map
$\Hom_{\K(\mathcal{A})}(K,I)\ra\Hom_{\D(\mathcal{A})}(K,I)$ is an
isomorphism.

\section{A negative result}
\label{sec:negative-result}

Recall that the following conditions on a ring $R$ are equivalent
(see \cite[Theorem~1.50]{nicholson-yousif-quasi-frobenius},
\cite[(15.9)~Theorem]{lam-lectures-modules-rings}):
\begin{itemize}
\item
  $R$ is quasi-Frobenius, i.\,e.\ left and right artinian and
  left and right self-injective;
\item
  $R$ is left or right artinian, and left or right
  self-injective; 
\item
  $R$ is left or right noetherian, and left or right
  self-injective;
\item an
  arbitrary
  $R$-module is projective if and only if it is
  injective. 
\end{itemize}
If $R$ is a quasi-Frobenius ring, any $R$-module has projective
dimension either $0$ or $\infty$; hence the global dimension of
$R$ is either $0$ or $\infty$ (note that left and right global
dimension of $R$ coincide by \cite[Corollary~5]{auslander-dimension-III}).

Examples of quasi-Frobenius rings are $\bZ/n\bZ$ for any natural
number $n \geq 1$ and $k[\epsilon]/(\epsilon^n)$ for any field
$k$ and any natural number $n \geq 1$.  We first proved the
following theorem for the ring $k[\epsilon]/(\epsilon^2)$ of dual
numbers and advise the reader to keep this example in mind during
a first reading of its proof.

\begin{theorem}
  \label{t:self-injective-artin-gdim-infinite}
  Let $R$ be a quasi-Frobenius ring of infinite global dimension.
  Then the functor
  \begin{equation*}
    F \colon \D(\mod(R)) \ra
    \D(\Mod(R))    
  \end{equation*}
  is not full and the closure of its essential image under direct
  summands in $\D(\Mod(R))$ is strictly contained in
  $\D_{\mod(R)}(\Mod(R))$.  In particular, the induced functor
  $\D(\mod(R)) \ra \D_{\mod(R)}(\Mod(R))$ is not an equivalence.
\end{theorem}

\begin{proof}
  Since $R$ is not semisimple, there is a finitely generated
  $R$-module $S \in \mod(R)$ that is not projective (for example
  a non-projective quotient of $R$ or a non-projective simple
  subquotient of $R$).  Since $R$ is an artinian ring,
  $S$ has
  a
  minimal projective resolution
  \begin{equation*}
    \dots \ra P_{n+1} \xra{d_n} P_n \ra
    \dots \ra P_1
    \xra{d_0} P_0 \sra S \ra 0.
  \end{equation*}
  This just means that the above sequence is exact, that all
  $P_n$ are in $\proj(R)$, and that $P_0 \sra S$ and all the
  induced maps $P_{n+1} \sra U_n:=\image(d_n)$ are projective
  covers, for $n \in \bN$.  Note also that all $P_n$
  are non-zero because $S$ 
  has infinite projective dimension.
  
  Consider the complexes 
  \begin{align*}
    B_1
    =
    (\dots
    \xra{\phantom{d}}
    0
    \xra{\phantom{d}}
    \makebox[0pt][l]{$0$}\phantom{P_6}
    \xra{\phantom{d}}
    \makebox[0pt][l]{$0$}\phantom{P_5}
    \xra{\phantom{d}}
    P_2
    \xra{d}
    &
      P_1
      \xra{d}
      \makebox[0pt][l]{$P_0$}\phantom{P_1}
      \xra{\phantom{d}}
      \makebox[0pt][l]{$0$}\phantom{P_1}
      \xra{\phantom{d}}
      \makebox[0pt][l]{$0$}\phantom{P_0}
      \xra{\phantom{d}}
      0
      \xra{\phantom{d}}
      \dots),\\
    B_2
    =
    (\dots
    \xra{\phantom{d}}
    0
    \xra{\phantom{d}}
    \makebox[0pt][l]{$0$}\phantom{P_6}
    \xra{\phantom{d}}
    P_4
    \xra{d}
    P_3
    \xra{d}
    &
      P_2
      \xra{d}
      P_1
      \xra{d}
      \makebox[0pt][l]{$P_0$}\phantom{P_1}
      \xra{\phantom{d}}
      \makebox[0pt][l]{$0$}\phantom{P_0}
      \xra{\phantom{d}}
      0
      \xra{\phantom{d}}
      \dots),\\
    B_3
    =
    (\dots
    \xra{\phantom{d}}
    0
    \xra{\phantom{d}}
    P_6
    \xra{d}
    P_5
    \xra{d}
    P_4
    \xra{d}
    &
      P_3
      \xra{d}
      P_2
      \xra{d}
      P_1
      \xra{d}
      P_0
      \xra{\phantom{d}}
      0
      \xra{\phantom{d}}
      \dots),\\
    \dots
    \phantom{=
    (\dots
    \xra{\phantom{d}}
    0
    \xra{\phantom{d}}
    P_6
    \xra{d}
    P_5
    \xra{d}
    P_4
    \xra{d}}
    &
  \end{align*}
  in $\proj(R)$ where $B_n$ is concentrated in degrees $[-n,n]$
  with $P_n$ in degree zero.  Let
  \begin{equation*}
    M:= \bigoplus_{n \geq 1} B_n
  \end{equation*}
  be their coproduct in the abelian category of complexes and
  also in $\D(\Mod(R))$.  The cohomology of $M$ is $0$ in degree
  zero, $S$ in each positive degree and $U_{2n}$ in degree $-n$,
  for each $n \geq 1$. Hence $M \in \D_{\mod(R)}(\Mod(R))$.  We
  claim that $M$ is not in the closure of the essential image of
  $F$ under direct summands in $\D(\Mod(R))$.

  Consider the canonical injective map 
  \begin{equation*}
    \mu \colon M =\bigoplus_{n \geq 1} B_n \ra
    M':=\prod_{n \geq 1} B_n 
  \end{equation*}
  between the coproduct and the product of the complexes $B_n$.
  The cohomology of $M'$ coincides with that of $M$, and $\mu$ is
  a quasi-isomorphism and becomes an isomorphism in
  $\D(\Mod(R))$.

  Since all projective $R$-modules are injective, all $B_n$ are
  bounded complexes of projective and injective $R$-modules and
  therefore are h-projective and h-injective (both in
  $\K(\mod(R))$ and in $\K(\Mod(R))$).  Since $M$ is also the
  coproduct of the $B_n$ in the homotopy category $\K(\Mod(R))$,
  it is h-projective.  Since $M'$ is also the product of the
  $B_n$ in the homotopy category $\K(\Mod(R))$, it is
  h-injective.
  
  Assume that there is an object $N \in \D(\mod(R))$ such that
  $M$ is a direct summand of $F(N)$ in $\D(\Mod(R))$.  Then there
  are morphisms $\iota \colon M \ra F(N)$ and
  $\pi \colon F(N) \ra M$ in $\D(\Mod(R))$ such that
  $\pi \circ \iota=\id_M$.  Then
  \begin{equation*}
    \mu = \mu \circ \pi \circ \iota \colon
    M \xra{\iota} F(N) \xra{\mu \circ \pi} M'.
  \end{equation*}
  Since $M$ is h-projective and $M'$ is h-injective, $\iota$ and
  $\mu \circ \pi$ can be represented by honest morphisms of
  complexes. The composition of representing morphisms is then a
  morphism
  \begin{equation*}
    e \colon M \ra M'
  \end{equation*}
  of complexes which becomes equal to $\mu$ in $\D(\Mod(R))$ and
  has the property that the image $e(M^p)=e^p(M^p)$ is finitely
  generated as an $R$-module, for each $p \in \bZ$.
  
  Since $M$ is h-projective (or since $M'$ is h-injective) there
  is a homotopy $h \colon M \ra M'$ (of degree one) between $e$
  and $\mu$.  In particular, we have
  \begin{equation}
    \label{eq:e-mu=dh}
    e^0 - \mu^0 = d_{M'}^{-1} \circ h^0 + h^1 \circ d_M^0
  \end{equation}
  as maps $M^0 \ra M'^0$.

  Let $\rad(R)$ be the Jacobson radical of $R$ and let $\rad(T)$
  be the radical of an $R$-module $T$.  Then $R/\rad(R)$ is a
  semisimple ring and $\rad(T)=\rad(R)T$ (see
  \cite[Proposition~15.16,
  Corollary~15.18]{anderson-fuller-rings}).
  We write $\ol{T} := T/\rad(T) = T/\rad(R) T$ for the quotient
  by the radical and $\ol{f} \colon \ol{T} \ra \ol{T}'$ for the
  morphism induced by an $R$-module morphism $f \colon T \ra T'$.
  With this notation, all the differentials
  $d_n \colon P_{n+1} \ra P_n$ of our minimal projective
  resolution induce zero morphisms
  \begin{equation*}
    \ol{d}_n=0 \colon \ol{P}_{n+1}
    \ra 
    \ol{P}_{n}
  \end{equation*}
  for each $n \in \bN$ because our projective covers become
  isomorphisms when passing to the respective quotients modulo
  the radicals (see
  \cite[Proposition~I.4.3]{ARS-rep-artin-algebras} for finitely
  generated modules or more generally
  \cite[27.13]{anderson-fuller-rings}).  The equalities
  \begin{equation*}
    \ol{M}^0
    = \frac{M^0}{\rad(M^0)}
    = \frac{M^0}{\rad(R)M^0}
    = \frac{\bigoplus_{n \geq 1} P_n}
    {\bigoplus_{n \geq 1} \rad(R)P_n}
    =\bigoplus_{n \geq 1} \ol{P}_n
  \end{equation*}
  and similarly for $\ol{M}^1$ imply
  $\ol{d}^0_M=\bigoplus_{n \geq 1} \ol{d}_{n-1}=0$.

  We have
  $\rad(R)M'^0= \rad(R) \prod_{n \geq 1} P_n =\prod_{n \geq 1}
  \rad(R)P_n$ because $\rad(R)$ is finitely generated from the
  right.
  Therefore we obtain equalities 
  \begin{equation*}
    \ol{M'}^0
    = \frac{M'^0}{\rad(M'^0)}
    = \frac{M'^0}{\rad(R)M'^0}
    = \frac{\prod_{n \geq 1} P_n}
    {\prod_{n \geq 1} \rad(R)P_n}
    =\prod_{n \geq 1} \ol{P}_n
  \end{equation*}
  and similarly for $\ol{M'}^{-1}$. We deduce
  $\ol{d}^{-1}_{M'}=\prod_{n \geq 1} \ol{d}_n=0$.

  Hence \eqref{eq:e-mu=dh} yields
  \begin{equation*}
    \ol{e}^0 = \ol{\mu}^0
  \end{equation*}
  as maps $\ol{M}^0 \ra \ol{M'}^0$. The image
  of $\ol{e}^0$ is finitely generated as an $R$-module because
  $e^0(M^0)$ has this property. But the map
  $\ol{\mu}^0 \colon \ol{M}^0 \ra \ol{M'}^0$ is, under the above
  identifications, the canonical injective map
  $\bigoplus_{n \geq 1} \ol{P}_n \hra \prod_{n \geq 1} \ol{P}_n$
  whose image is not finitely generated because all $P_n$ and
  hence all $\ol{P}_n$ are non-zero.
  This contradiction proves the claim that
  $M \in \D_{\mod(R)}(\Mod(R))$ is not in the closure of the
  essential image of $F$ under direct summands in $\D(\Mod(R))$.

  Let us deduce from this that $F$ is not full.  Consider the
  truncation triangle
  \begin{equation*}
    \tau_{\leq 0} M \ra M \ra \tau_{\geq 1}M \xra{\delta}
    \Sigma \tau_{\leq 0} M 
  \end{equation*}
  with respect to the standard t-structure.  Clearly,
  $\tau_{\geq 1} M \cong \bigoplus_{n \geq 1} \Sigma^{-n} S =:
  T_1$ and
  $\tau_{\leq 0} M \cong \bigoplus_{n \geq 1} \Sigma^{n} U_{2n}
  =: T_0$ are in the essential image of $F$.  Assuming that $F$
  is full, there is a morphism
  \begin{equation*}
    \delta' \colon T_1 \ra \Sigma T_0
  \end{equation*}
  in $\D(\mod(R))$ whose image $F(\delta')$ coincides with
  $\delta$ modulo the above isomorphisms. Complete $\delta'$ to a
  triangle $T_0 \ra E \ra T_1 \xra{\delta'} \Sigma T_0$. The
  image of this triangle under $F$ is isomorphic to the above
  truncation triangle.  In particular $F(E) \cong M$. This
  contradicts the fact proven above that $M$ is not in the
  essential image of $F$. Hence $F$ is not full.
\end{proof}

We do not know whether the functor in
Theorem~\ref{t:self-injective-artin-gdim-infinite} is faithful,
not even for $R=\bC[\epsilon]/(\epsilon^2)$.  We conclude this
section with several remarks, some of which may help to solve
this question.
\begin{enumerate}
\item
  \label{enum:Dplus-equi}
  Let $R$ be a left noetherian ring with the following property:
  \begin{enumerate}[label=($\star$)]
  \item
    \label{enum:mono-to-fg-quotient}
    If $N \hra M$ is an arbitrary monomorphism of $R$-modules
    with $N$ finitely generated, then there is
    a morphism
    $M \ra Q$
    such that $Q$ is finitely generated and such that the
    composition $N \hra M \ra Q$ is injective.
  \end{enumerate}
  Then the functor induced by $F$ is an equivalence
  \begin{equation}
    \label{eq:D-plus-mod-Mod}
    \D^+(\mod(R)) \sira \D^+_{\mod(R)}(\Mod(R)).
  \end{equation}
  To prove this, just dualize the proof of
  Proposition~\ref{p:D-minus-equivalence}.
\item
  \label{enum:qF-has-property}
  Any quasi-Frobenius ring $R$ satisfies
  property~\ref{enum:mono-to-fg-quotient}.
  Indeed, let $N \hra M$ be a monomorphism of $R$-modules with
  $N$ finitely generated.  Since $M$ has an injective hull and
  any injective $R$-module is projective, there is a monomorphism
  $M \hra \bigoplus_{i \in I} R$ for some index set $I$. Let
  $J \subset I$ be a finite set such that the image of
  $N \hra M \hra \bigoplus_{i \in I} R$ lands in
  $\bigoplus_{j \in J} R$. Define
  $U:= M \cap \bigoplus_{i \in I\setminus J} R$. Then
  $M \sra M/U$ does the job.
\item
  \label{enum:partial-ff}
  Let $R$ be a quasi-Frobenius ring (of infinite global dimension
  -- the case that $R$ is semisimple is boring).  Note that we
  have equivalences \eqref{eq:D-minus-mod-Mod} and
  \eqref{eq:D-plus-mod-Mod} by the previous items
  \ref{enum:Dplus-equi} and \ref{enum:qF-has-property}.
  Intelligent truncation and the fact that $F$ is t-exact with
  respect to the standard t-structures then shows that
  \begin{equation*}
    F
    \colon
    \Hom_{\D(\mod(R))}(X,Y)
    \ra 
    \Hom_{\D(\Mod(R))}(FX,FY)
  \end{equation*}
  is bijective in the following two cases:
  \begin{itemize}
  \item $X \in \D^-(\mod(R))$ and $Y \in \D(\mod(R))$;
  \item $X \in \D(\mod(R))$ and $Y \in \D^+(\mod(R))$.
  \end{itemize}
  For bijectivity in the first case
  it is enough to assume that $R$ is left coherent.   
\item
  \label{enum:factor-non-ff-witness}
  Let $R$ be a quasi-Frobenius ring.  If there is a non-zero
  morphism $f \colon X \ra Y$ in $\D(\mod(R))$ that becomes zero
  in $\D(\Mod(R))$ (i.\,e.\ if $F$ is not faithful), then, for
  all integers $m,n \in \bZ$, the morphism $f$ factors as
  \begin{equation*}
    X \ra \tau_{\geq m}X \xra{g} \tau_{\leq n} Y \ra Y
  \end{equation*}
  in $\D(\mod(R))$, where $g$ is some (obviously non-zero)
  morphism that becomes zero in $\D(\Mod(R))$.

  \begin{proof}
    Consider the intelligent truncation triangle
    \begin{equation*}
      \tau_{<m} X \xra{u} X \xra{v} \tau_{\geq m}X \xra{d}
      \Sigma \tau_{<m} X.
    \end{equation*}
    Applying the functor $\Hom(-,Y)$ to this triangle and
    $\Hom(-,FY)$ to its image under $F$ yields the commutative
    diagram
    \begin{equation*}
      \xymatrix{
        {(\tau_{<m} X, Y)} \ar[d]_-{F}^-{\sim} &
        \ar[l]_-{u^*} {(X,Y)} \ar[d]_-{F} &
        \ar[l]_-{v^*} {(\tau_{\geq m}X,Y)} \ar[d]_-{F} &
        \ar[l]_-{d^*}
        {(\Sigma \tau_{<m} X,Y)} \ar[d]_-{F}^-{\sim}\\
        {(F\tau_{<m} X, FY)} &
        \ar[l]_-{(Fu)^*} {(FX,FY)} &
        \ar[l]_-{(Fv)^*} {(F\tau_{\geq m}X,FY)} &
        \ar[l]_-{(Fd)^*}
        {(\Sigma F\tau_{<m} X,FY)}
      }
    \end{equation*}
    with exact rows.
    The leftmost and rightmost vertical arrows are bijective by
    \ref{enum:partial-ff}. Let $f \in \Hom(X,Y)$ be non-zero with
    $F(f)=0$.  The obvious diagram chase shows that there is some
    morphism $\hat{f} \in \Hom(\tau_{\geq m}X, Y)$ with
    $\hat{f} \circ v=v^*(\hat{f})=f$. Then $F(\hat{f})$ goes to
    zero under $(Fv)^*$ and another diagram chase yields some
    morphism $h \in \Hom(\Sigma\tau_{<m}X,Y)$ such that
    $F(d^*(h))=(F(d))^*(F(h))= F(\hat{f})$.  Hence
    $F(\hat{f}-h \circ d)$ is zero and the composition
    $X \xra{v} \tau_{\geq m} X \xra{\hat{f}-h \circ d} Y$ is $f$.

    The same argument, now starting with the truncation triangle
    for $Y$ and the map $\hat{f}-h \circ d$, finishes the proof.
  \end{proof}
\item
  Let $R$ be a quasi-Frobenius ring. Then 
  $\D(\mod(R))$ is idempotent-complete.

  \begin{proof}
    Let $e \colon E \ra E$ be an idempotent endomorphism in
    $\D(\mod(R))$.  Since intelligent truncation is functorial,
    $\tau_{\leq 0}(e)$ and $\tau_{>0}(e)$ are idempotent
    endomorphisms.  By
    \cite[Proposition~2.3]{le-chen-karoubi-trcat-bdd-t-str} it is
    enough to show that they split.  Since $\mod(R)$ has enough
    projectives and injectives (the injective hull of a finitely
    generated $R$-module is projective and then easily seen to be
    finitely generated), the obvious functors
    $\K^-(\proj(R)) \ra \D^-(\mod(R))$ and
    $\K^+(\inj(R)) \ra \D^+(\mod(R))$ are equivalences of
    triangulated categories, where $\K^-$ resp.\ $\K^+$ indicate
    that only bounded above resp.\ bounded below complexes are
    considered.  Now note that $\K^-(\proj(R))$ and
    $\K^+(\inj(R))$ are idempotent-complete by
    \cite[Theorem~3.1]{schnuerer-idempotent}.
  \end{proof}
\end{enumerate}

\section{First positive result}
\label{sec:first-positive-result}

Let $\mathcal{A}$ be an abelian category. The \define{global
  dimension of $\mathcal{A}$} is the smallest integer $d \geq -1$
with the property that $\Ext_{\mathcal{A}}^{d+1}(M,N)$ vanishes
for all objects $M,N \in\mathcal{A}$, if such an integer exists,
and $\infty$ otherwise.

\begin{proposition}
  \label{p:finite-global-dim-proj-comp=>h-proj}
  Let $\mathcal{A}$ be an abelian category with enough projective
  objects that has finite global dimension.  Then
  \begin{enumerate}
  \item
    \label{enum:h-proj-resolution}
    for any complex $K$ in $\mathcal{A}$ there is a
    quasi-isomorphism $P \ra K$ where $P$ has projective
    components and hence is h-projective by
    \ref{enum:projective-components-h-projective};
  \item
    \label{enum:projective-components-h-projective}
    any complex in $\mathcal{A}$ with projective components is
    h-projective.
  \end{enumerate}
  In particular, the obvious functor is an equivalence
  $\K(\proj(\mathcal{A})) \sira \D(\mathcal{A})$.  Hence
  $\D(\mathcal{A})$ lives in the same universe as
  $\K(\proj(\mathcal{A}))$.
\end{proposition}

\begin{proof}
  Let $d \in \bN \cup \{-1\}$ be the global dimension of
  $\mathcal{A}$.  If $d=-1$ then $\mathcal{A}=0$ and all claims
  are trivial.  If $d \geq 0$ we will use the following
  well-known fact: If
  $0 \ra L \ra Q_{d-1} \ra \dots \ra Q_0 \ra M \ra 0$ is an exact
  sequence in $\mathcal{A}$ where all $Q_i$ are projective, then
  $L$ is projective. In particular, any object $M$ of
  $\mathcal{A}$ has a projective resolution of the form
  $0 \ra Q_d \ra Q_{d-1} \ra \dots \ra Q_0 \ra M \ra 0$.
  
  \ref{enum:h-proj-resolution}
  Our proof
  closely
  follows
  \cite[Appendix]{Keller-construction-of-triangle-equiv}.
  Let $K$ be a complex in $\mathcal{A}$.
  Since any object of $M$ has a projective
    resolution of length $d$, there is a
    projective resolution
  \begin{equation*}
    0 \ra P_d \ra P_{d-1} \ra \dots \ra P_0 \ra K \ra 0
  \end{equation*}
  of $K$ in the sense of Cartan--Eilenberg.
  To construct such a resolution, choose
  projective resolutions
  $0 \ra R_d^i \ra \dots \ra R_0^i \ra \HH^i(K) \ra 0$ and
  $0 \ra Q_d^i \ra \dots \ra Q_0^i \ra \B^i(K) \ra 0$ for all
  $i \in \bZ$ and combine these resolutions of all cohomology and
  coboundary objects in the obvious way to projective resolutions
  of all cocycle objects $\Z^i(K)$ and then to projective
  resolutions of all components $K^i$.
  
  Let $P$ be the total complex of
  $0 \ra P_d \ra P_{d-1} \ra \dots \ra P_0 \ra 0$. Then all
  components of $P$ are projective and the obvious map $P \ra K$
  is clearly a quasi-isomorphism.  This proves
  \ref{enum:h-proj-resolution}.

  We now show that $P$ is h-projective.
  Each complex $P_p$ has the form
  \begin{equation*}
    \dots \ra
    Q^i_p \oplus R^i_p \oplus Q^{i+1}_p
    \xra{d}
    Q^{i+1}_p \oplus R^{i+1}_p \oplus Q^{i+2}_p
    \ra \dots
  \end{equation*}
  where the differential $d$ is the composition of the projection
  to and the inclusion from $Q^{i+1}_p$. Hence $P_p$ is
  isomorphic in $\K(\mathcal{A})$ to the complex with vanishing
  differentials whose component in degree $i$ is the projective
  object $R^i_p$; this latter complex is certainly
  h-projective. Hence $P_p$ is h-projective.

  If $P_{\leq l}$ denotes the total complex of
  $0 \ra P_l \ra \dots \ra P_0 \ra 0$ we have triangles
  $P_{\leq l} \ra P_{\leq l+1} \ra \Sigma^{l+1} P_{l+1} \ra$ in
  $\K(\mathcal{A})$ for all $l$ with $0 \leq l \leq d$. Since
  $P_0=P_{\leq 0}$ and all $P_p$ are h-projective we see by
  induction that all $P_{\leq l}$ are h-projective.  In
  particular, $P=P_{\leq d}$ is h-projective.

  \ref{enum:projective-components-h-projective}
  Let $Q$ be a complex in $\mathcal{A}$ with projective
  components. We just proved that there is a quasi-isomorphism
  $f \colon P \ra Q$ where $P$ has projective components and is
  h-projective.  It is enough to show that $f$ is an isomorphism
  in $\K(\mathcal{A})$. Equivalently, we show that the standard
  cone $C$ of $f$ is zero in $\K(\mathcal{A})$. Note that $C$ is
  an acyclic complex with projective components. These two
  properties will imply that $C$ is isomorphic to $0$ in
  $\K(\mathcal{A})$.  Consider the exact sequence
  \begin{equation*}
    0 \ra \Z^{-d+1}(C) \ra C^{-d+1} \ra \dots \ra C^{-1} \ra C^{0}
    \ra \Z^1(C) 
    \ra 0 
  \end{equation*}
  in $\mathcal{A}$ where $\Z^i(C)$ is the $i$-th cocycle
  object. Since all $C^i$ are projective and $d$ is the global
  dimension of $\mathcal{A}$, the cocycle object $\Z^{-d+1}(C)$
  is projective. The obvious shift of this argument shows that
  all cocycle objects $\Z^i(C)$ are projective.  But then all
  short exact sequences $\Z^i(C) \hra C^i \sra \Z^{i+1}(C)$ split
  and it is easy to see that $C$ is contractible, i.\,e.\ zero in
  $\K(\mathcal{A})$.

  This proves \ref{enum:h-proj-resolution} and
  \ref{enum:projective-components-h-projective}. The remaining
  statements then follow from general properties of h-projective
  objects explained at the end of section~\ref{sec:notat-prel}.
\end{proof}

\begin{proposition}
  \label{p:essentially-surjective}
  Let $\mathcal{A}$ be an abelian category with exact countable
  coproducts, and let $\mathcal{U}\subset\mathcal{A}$ be a
  non-empty, full subcategory closed under kernels, cokernels,
  and extensions.  Assume that the abelian category $\mathcal{U}$
  has finite global dimension, and that for any epimorphism
  $A\sra U$ in $\mathcal{A}$ with $U\in\mathcal{U}$ there exists
  a morphism $V\ra A$ with $V\in\mathcal{U}$ such that the
  composition $V\ra A\sra U$ is an epimorphism.  Then the functor
  $\D(\mathcal{U})\ra\D_{\mathcal{U}}(\mathcal{A})$ is
  essentially surjective.
\end{proposition}

\begin{proof}
  Let $M$ be an object of
  $\D_{\mathcal{U}}(\mathcal{A})$. Consider its intelligent
  truncation
  $\tau_{\leq 0} M \in \D_{\mathcal{U}}^{\leq 0}(\mathcal{A})$.
  Then the proof of Proposition~\ref{p:D-minus-equivalence} shows
  that there is a complex $L_0\in\D^{\leq 0}(\mathcal{U})$
  together with a quasi-isomorphism
  $f_0 \colon L_0 \ra \tau_{\leq 0} M$ of complexes in
  $\mathcal{A}$.
  
  Let $d \in \bN \cup \{-1\}$ be the global dimension
  of~$\mathcal{U}$.  We explain the following diagram in
  $\D(\mathcal{A})$ --- some parts come from honest morphisms and
  inclusions of complexes.
  \begin{equation}
    \label{eq:morphism-of-triangles}
    \begin{gathered}
      \xymatrix{
        {\Sigma^{-2}\HH^1(M)} \ar[r] &
        {\tau_{\leq 0} M} \ar[r]^-{\mu_0} &
        {\tau_{\leq 1} M} \ar[r] &
        {\Sigma^{-1}\HH^1(M)}\\
        {\Sigma^{-2} Q_1}
        \ar[r]^-{g_1} \ar[u]^-{\Sigma^{-2} q_1} &
        {L_0} \ar[r]^-{\iota_0} \ar[u]^-{f_0} &
        {L_1} \ar[r] \ar[u]^-{f_1}&
        {\Sigma^{-1} Q_1}  \ar[u]^-{\Sigma^{-1} q_1}
      }
    \end{gathered}
  \end{equation}
  The upper row is a rotation of the obvious truncation triangle
  in $\D(\mathcal{A})$ where $\mu_0$ is the obvious inclusion of
  complexes.  To construct the lower row, recall that
  $f_0\colon L_0\ra\tau_{\le0}M$ is an isomorphism in
  $\D(\mathcal{A})$.  Therefore, there exists a morphism
  $h_1\colon\Sigma^{-2}\HH^1(M)\ra L_0$ in $\D(\mathcal{A})$
  such that the map $\Sigma^{-2}\HH^1(M)\ra\tau_{\le0}M$ from
  the upper row of \eqref{eq:morphism-of-triangles} factors as
  $\Sigma^{-2}\HH^1(M)\xra{h_1}L_0\xra{f_0}\tau_{\le0}M$.
  The morphism $h_1$ can be represented by a roof consisting of
  a quasi-isomorphism $\Sigma^{-2}K_1\ra\Sigma^{-2}\HH^1(M)$ and
  a morphism of complexes $\Sigma^{-2}K_1\ra L_0$.
  Using the argument from the proof of
  Proposition~\ref{p:D-minus-equivalence}, we can assume that
  $K_1\in\D^{\leq 0}(\mathcal{U})$ (since
  $\HH^1(M)\in\mathcal{U}$).
  
  Our next aim is to show that the complex $K=K_1$ can be
  replaced by a bounded complex with terms in the cohomological
  degrees from $-d$ to $0$.  Indeed, let $\Z^{-d}(K)$ be the
  image of the differential $K^{-d-1} \ra K^{-d}$.  Then the
  exact sequence
  $0\ra\Z^{-d}(K)\ra K^{-d}\ra K^{-d+1} \ra\dots\ra K^1\ra K^0\ra
  \HH^0(K)\ra0$ (where
  $\HH^0(K) \cong \HH^1(M)$)
  represents a Yoneda
  extension class in the group
  $\Ext^{d+1}_{\mathcal{U}}(\HH^0(K),\Z^{-d}(K))$.  Since
  $\mathcal{U}$ has global dimension $d$, this Ext group
  vanishes.  By construction of Yoneda Ext, this means that there
  is a roof connecting our exact sequence with the exact sequence
  $0\ra\Z^{-d}(K)\xra{\id}\Z^{-d}(K)\ra0\ra\dots\ra0\ra
  \HH^0(K)\xra{\id} \HH^0(K)\ra0$ representing the trivial
  extension class.  So we have the following commutative diagram
  with exact rows in the category $\mathcal{U}$.
  \begin{equation*}
    \xymatrix@C-=0.7cm{
      0\ar[r] & \Z^{-d}(K) \ar[r] & K^{-d} \ar[r] & K^{-d+1} \ar[r]
      & {\dots} \ar[r] & K^{-1} \ar[r] & K^0 \ar[r] & \HH^0(K)
      \ar[r] 
      & 0 \\
      0\ar[r] & \Z^{-d}(K) \ar[r]\ar@{=}[u]\ar@{=}[d] &
      \widetilde Q^{-d} \ar[r]\ar[u]\ar[d] & Q^{-d+1}
      \ar[r]\ar[u]\ar[d] 
      & {\dots} \ar[r] & Q^{-1} \ar[r]\ar[u]\ar[d] & Q^0
      \ar[r]\ar[u]\ar[d] & \HH^0(K) \ar[r]\ar@{=}[d]\ar@{=}[u] &
      0 \\ 
      0 \ar[r] & \Z^{-d}(K)
      \ar[r]^-{\id}
      & \Z^{-d}(K) \ar[r] & 0
      \ar[r] 
      & {\dots} \ar[r] & 0 \ar[r] & \HH^0(K)
      \ar[r]^-{\id}
      &
      \HH^0(K) \ar[r] & 0
    }
  \end{equation*}
  It is clear from this diagram that the monomorphism
  $\Z^{-d}(K)\ra\widetilde Q^{-d}$ splits; so
  $\widetilde Q^{-d}\cong\Z^{-d}(K)\oplus Q^{-d}$ for some object
  $Q^{-d}\in\mathcal{U}$.  We have constructed a bounded complex
  $Q:=(Q^{-d}\ra Q^{-d+1}\ra\dots \ra Q^0)$ in $\mathcal{U}$
  together with a quasi-isomorphism $Q\ra K$.
  
  Now we can finish our construction of diagram
  \eqref{eq:morphism-of-triangles}.  Put $Q_1=Q$, and let
  $q_1\colon Q_1\ra\HH^1(M)$ be the composition
  $Q_1\ra K_1\ra\HH^1(M)$
  of quasi-isomorphisms
  of complexes.
  Furthermore, let $g_1\colon\Sigma^{-2}Q_1\ra L_0$ be the
  composition
  $\Sigma^{-2}Q_1\ra\Sigma^{-2}K_1\ra L_0$
  of morphisms of complexes.
  Then the leftmost
  square in \eqref{eq:morphism-of-triangles} is
  commutative in $\D(\mathcal{A})$.
  We define $L_1$
  as the standard cone of $g_1$; in particular,
  $L_0 \xra{\iota_0} L_1 \ra\Sigma^{-1}Q_1$ is a short exact
  sequence of complexes that splits in each degree; we have
  $L_1=L_0 \oplus \Sigma^{-1}Q_1$ if we ignore the differentials.
  The lower row of our diagram is
  then
  the triangle
    coming from the obvious standard triangle in $\K(\mathcal{U})$.
  The
  morphism $f_1$ is a (non-unique) morphism in $\D(\mathcal{A})$
  completing the two morphisms $\Sigma^{-2}q_1$ and $f_0$ to a
  morphism of triangles. Since $q_1$ and $f_0$ are isomorphisms,
  $f_1$ is an isomorphism in $\D(\mathcal{A})$.
  
  Next we are going to apply the same construction to the
  intelligent truncation triangle
  $\Sigma^{-3}\HH^2(M)\ra\tau_{\leq 1}M\ra \tau_{\leq
    2}M\ra\Sigma^{-2}\HH^2(M)$ and the morphism
  $f_1\colon L_1\ra \tau_{\leq 1}M$.  At this stage, unlike
  $f_0$, the morphism $f_1$ only exists in the derived category
  $\D(\mathcal{A})$; we do not know whether $f_1$ comes from an
  honest morphism of complexes.  But this is not a problem for
  our construction.  Since $f_1$ is an isomorphism in
  $\D(\mathcal{A})$, there exists a morphism
  $h_2\colon\Sigma^{-3}\HH^2(M)\ra L_1$ in $\D(\mathcal{A})$
  such that the map $\Sigma^{-3}\HH^2(M)\ra\tau_{\le1}M$
  factors as
  $\Sigma^{-3}\HH^2(M)\xra{h_2}L_1\xra{f_1}\tau_{\le1}M$.
  The morphism $h_2$ can be represented by a roof consisting of a
  quasi-isomorphism $\Sigma^{-3}K_2\ra\Sigma^{-3}\HH^2(M)$ and a
  morphism of complexes $\Sigma^{-3}K_2\ra L_1$, etc.
  
  Iterating this construction yields for each $n \in \bN$ a
  diagram
  \begin{equation*}
    \xymatrix{
      {\Sigma^{-(n+1)-1}\HH^{n+1}(M)} \ar[r] &
      {\tau_{\leq n} M} \ar[r]^-{\mu_n} &
      {\tau_{\leq {n+1}} M} \ar[r] &
      {\Sigma^{-(n+1)}\HH^{n+1}(M)}\\
      {\Sigma^{-(n+1)-1} Q_{n+1}}
      \ar[r]^-{g_{n+1}}
      \ar[u]^-{\Sigma^{-(n+1)-1} q_{n+1}} & 
      {L_n} \ar[r]^-{\iota_n} \ar[u]^-{f_n} &
      {L_{n+1}} \ar[r] \ar[u]^-{f_{n+1}}&
      {\Sigma^{-(n+1)} Q_{n+1}.}
      \ar[u]^-{\Sigma^{-(n+1)} q_{n+1}}
    }
  \end{equation*}
  with properties similar to those of
  \eqref{eq:morphism-of-triangles}.

  Note that $M$ is the colimit of
  $\tau_{\leq 0} M \subset \tau_{\leq 1} M \subset \dots$ in the
  category of complexes in $\mathcal{A}$.  Let $L_\infty$ be the
  colimit of
  $L_0 \overset{\iota_0}{\subset} L_1 \overset{\iota_1}{\subset}
  \dots$ in the category of complexes in $\mathcal{A}$.  Then, if
  we ignore the differentials,
  \begin{equation*}
    L_\infty = L_0 \oplus \Sigma^{-1}Q_1 \oplus
    \Sigma^{-2}Q_2 \oplus \ldots
    = L_0 \oplus \bigoplus_{n \geq 1} \Sigma^{-n}Q_n.
  \end{equation*}
  Since $L_0$ and all $Q_n$ are complexes in $\mathcal{U}$ and
  $\Sigma^{-n}Q_n$ is concentrated in degrees $[-d+n,n]$, we see
  that $L_\infty$ is a complex in $\mathcal{U}$ as well.

  The sequence
  $\tau_{\leq 0} M \subset \tau_{\leq 1} M \subset \dots$ of
  inclusions of complexes is stabilizing in every degree, and the
  same statement is true for the sequence
  $L_0 \overset{\iota_0}{\subset} L_1 \overset{\iota_1}{\subset}
  \dots$ of inclusions of complexes (in the latter case all
  inclusions are even degreewise split).  Using the existence of
  countable coproducts in $\mathcal{A}$, we have degreewise split
  short exact sequences
  \begin{equation*}
    \bigoplus_{n \in \bN} \tau_{\leq n} M
    \hra 
    \bigoplus_{n \in \bN} \tau_{\leq n} M \sra M
    \qquad \text{and} \qquad
    \bigoplus_{n \in \bN} L_n
    \hra 
    \bigoplus_{n \in \bN} L_n \sra L_\infty  
  \end{equation*}
  in the abelian category of complexes in $\mathcal{A}$ (where
  the first map is ``identity minus shift'' in both cases).
  These two short exact sequences of complexes give rise to the
  two horizontal triangles in $\D(\mathcal{A})$ in the following
  diagram.
  \begin{equation*}
    \xymatrix{
      {\bigoplus_{n \in \bN} \tau_{\leq n} M}
      \ar[r]
      &
      {\bigoplus_{n \in \bN} \tau_{\leq n} M}
      \ar[r]
      &
      {M}
      \ar[r]
      &
      {\Sigma\bigoplus_{n \in \bN} \tau_{\leq n} M}
      \\
      {\bigoplus_{n \in \bN} L_n}
      \ar[r]
      \ar[u]^-{\bigoplus f_n}
      &
      {\bigoplus_{n \in \bN} L_n}
      \ar[r]
      \ar[u]^-{\bigoplus f_n}
      &
      {L_\infty}
      \ar[r]
      \ar@{..>}[u]^-{f}
      &
      {\Sigma\bigoplus_{n \in \bN} L_n}
      \ar[u]^-{\Sigma \bigoplus f_n}
    }
  \end{equation*}
  Our assumption that $\mathcal{A}$ has exact countable
  coproducts implies that countable coproducts in
  $\D(\mathcal{A})$ exist and can be computed naively.  Hence
  these two triangles exhibit $M$ and $L_\infty$ as homotopy
  colimits.  Moreover, taking the coproduct of the morphisms
  $f_n$ in $\D(\mathcal{A})$ defines the vertical morphism
  $\bigoplus f_n$.  Note that the square on the left is
  commutative in $\D(\mathcal{A})$ because
  $\mu_n \circ f_n=f_{n+1} \circ \iota_n$ for all $n \in \bN$.
  Hence a dotted morphism $f$ exists (non-uniquely) completing
  this square to a morphism of triangles.  This morphism $f$ is
  an isomorphism because all $f_n$ and hence their coproduct
  $\bigoplus f_n$ are isomorphisms.  Since $L_\infty$ is a
  complex in $\mathcal{U}$ we deduce that $M$ is in the essential
  image of our functor
  $\D(\mathcal{U}) \ra \D_{\mathcal{U}}(\mathcal{A})$.
\end{proof}

\begin{theorem}
  \label{t:equiv-abelian-subcat-gdim-finite}
  Let $\mathcal{A}$ be an abelian category of finite global
  dimension with exact countable coproducts, and let
  $\mathcal{U}\subset\mathcal{A}$ be a non-empty, full
  subcategory closed under kernels, cokernels, and extensions.
  Assume that both abelian categories $\mathcal{A}$ and
  $\mathcal{U}$ have enough projective objects and that
  $\proj(\mathcal{U}) \subset \proj(\mathcal{A})$ (and hence
  $\proj(\mathcal{U})= \mathcal{U} \cap \proj(\mathcal{A})$).
  Then the obvious functor is an equivalence
  $\D(\mathcal{U}) \sira \D_\mathcal{U}(\mathcal{A})$ of
  triangulated categories.
\end{theorem}

\begin{remark}
  If $\mathcal{A}$ is hereditary, i.\,e.\ of global dimension
  $\leq 1$, Theorem~\ref{t:equiv-abelian-subcat-gdim-finite} can
  be strengthened to Corollary~\ref{c:equiv-hereditary}.  Neither
  the assumption that countable coproducts exist nor the
  assumptions on projective objects are needed in this case.
\end{remark}
  
\begin{proof}
  Since there are enough projectives in $\mathcal{U}$ and they
  are also projective in $\mathcal{A}$, it is clear that the
  inclusion $\mathcal{U}\ra\mathcal{A}$ induces isomorphisms on
  all
  Ext groups.  Therefore, the global dimension of
  $\mathcal{U}$ cannot exceed the global dimension of
  $\mathcal{A}$.

  The inclusion $\mathcal{U}\ra\mathcal{A}$ restricts to a fully
  faithful functor $\proj(\mathcal{U})\ra\proj(\mathcal{A})$.
  Hence the induced triangulated functor
  $\K(\proj(\mathcal{U})) \ra \K(\proj(\mathcal{A}))$ is fully
  faithful.

  Since the category $\mathcal{A}$ has finite global dimension by
  assumption,
  Proposition~\ref{p:finite-global-dim-proj-comp=>h-proj} can be
  applied both to $\mathcal{U}$ and $\mathcal{A}$.  Hence, the
  vertical functors in the obvious commutative diagram
  \begin{equation*}
    \xymatrix{
      {\K(\proj(\mathcal{U}))} \ar[r] \ar[d]_-{\sim} & 
      {\K(\proj(\mathcal{A}))} \ar[d]_-{\sim} \\
      {\D(\mathcal{U})} \ar[r] & 
      {\D(\mathcal{A})}
    }
  \end{equation*}
  are equivalences. Since the upper horizontal functor is fully
  faithful, the lower horizontal functor has the same property.
  It clearly lands in $\D_{\mathcal{U}}(\mathcal{A})$.  So it
  remains to be shown that its essential image is
  $\D_{\mathcal{U}}(\mathcal{A})$.

  Let $A\sra U$ be an epimorphism in $\mathcal{A}$ with
  $U\in\mathcal{U}$.  Choose a projective object
  $V\in\proj(\mathcal{U})$ together with an epimorphism
  $V\sra U$.  Since $V$ is also projective in $\mathcal{A}$, the
  epimorphism $V\sra U$ factors through the epimorphism
  $A\sra U$.

  Therefore, Proposition \ref{p:essentially-surjective} is
  applicable, and we are done.
\end{proof}

\begin{corollary}
  \label{c:equiv-Dmod-DMod-gdim-finite}
  Let $R$ be a left coherent ring of finite left global
  dimension.
  Then the obvious triangulated functor is an equivalence
  \begin{equation*}
    \D(\mod(R)) \sira \D_{\mod(R)}(\Mod(R)).
  \end{equation*}
\end{corollary}

\begin{proof}
  This is a particular case of
  Theorem~\ref{t:equiv-abelian-subcat-gdim-finite}.  Note that
  $\mod(R)$ has enough projective objects since $R \in \mod(R)$.
\end{proof}

\section{Second positive result}
\label{sec:second-positive-result}

Let $\mathcal{A}$ be an abelian category.  Any short exact
sequence $0\ra K\ra L\ra M\ra0$ of complexes in $\mathcal{A}$ can
be considered as a bicomplex with three rows.
\begin{equation*}
  \xymatrix{ 
    {\dots} \ar[r] & K^{-1} \ar[r]\ar[d] & K^0 \ar[r]\ar[d]
    & K^1 \ar[r]\ar[d] & K^2 \ar[r]\ar[d] & {\dots} \\
    {\dots} \ar[r] & L^{-1} \ar[r]\ar[d] & L^0 \ar[r]\ar[d]
    & L^1 \ar[r]\ar[d] & L^2 \ar[r]\ar[d] & {\dots} \\
    {\dots} \ar[r] & M^{-1} \ar[r] & M^0 \ar[r]
    & M^1 \ar[r] & M^2 \ar[r] & {\dots}
  }
\end{equation*}
We consider the total complex of this bicomplex and denote it by
$\Tot(K\ra L\ra M)$.

A complex in $\mathcal{A}$ is said to be \define{absolutely
  acyclic} if it belongs to the minimal thick subcategory of
$\K(\mathcal{A})$ containing all the total complexes
$\Tot(K\ra L\ra M)$ of short exact sequences
$0\ra K\ra L\ra M\ra0$ of complexes in $\mathcal{A}$.  Clearly,
any absolutely acyclic complex is acyclic, but the converse is
not true in general.  One can easily see that any bounded acyclic
complex is absolutely acyclic.

\begin{example}
  \label{ex:acyclic-not-absolutely}
  The unbounded acyclic complex
  \begin{equation*}
    \dots\xra{\epsilon} R\xra{\epsilon} R\xra{\epsilon} R
    \xra{\epsilon}\dots
  \end{equation*}
  of projective-injective modules over the ring of dual numbers
  $R=k[\epsilon]/(\epsilon^2)$ (over any field~$k$) is not
  absolutely acyclic.  Moreover, both intelligent truncations
  \begin{equation*}
    \dots\xra{\epsilon} R\xra{\epsilon} R\xra{\epsilon} R
    \sra k\ra 0 \qquad \text{and} \qquad
    0\ra k \hra R \xra{\epsilon} R\xra{\epsilon} R
    \xra{\epsilon}\dots
  \end{equation*}
  of this complex
  are acyclic, but not absolutely acyclic.
  Similarly, the unbounded acyclic complex
  \begin{equation*}
    \dots\xra{2} R\xra{2} R\xra{2} R
    \xra{2}\dots
  \end{equation*}
  of projective-injective modules over the ring $R=\bZ/4\bZ$ is
  not absolutely acyclic.  The intelligent truncations
  \begin{equation*}
    \dots\xra{2} R\xra{2} R\xra{2} R
    \sra\bZ/2\bZ\ra 0 \qquad \text{and} \qquad
    0\ra \bZ/2\bZ \hra R \xra{2} R\xra{2} R
    \xra{2}\dots
  \end{equation*}
  of this complex
  are again acyclic, but not absolutely acyclic.

  More generally, let $R$ be a quasi-Frobenius ring and let $M$
  be a non-projective (and hence non-injective) $R$-module.  Let
  $\ldots\ra P_2\ra P_1\ra P_0 \xsra{\pi} M \ra 0$ be a
  projective resolution of $M$ and let
  $0\ra M \xhra{\iota} J^0\ra J^1\ra J^2\ra\ldots$ be an
  injective resolution of $M$.  Let $f$ be the composition
  $P_0 \xsra{\pi} M \xhra{\iota} J^0$.  Then the three acyclic
  complexes
  \begin{align}
    \label{eq:unbounded}
    \ldots\ra P_2\ra P_1\ra P_0 \xra{f}
    & J^0\ra J^1\ra J^2\ra\ldots,\\
    \label{eq:bounded-above}
    \ldots\ra P_2\ra P_1\ra P_0\xsra{\pi}
    & M\ra 0,\\
    \label{eq:bounded-below}
    0\ra M\xhra{\iota}
    & J^0\ra J^1\ra J^2\ra\ldots 
  \end{align}
  of $R$-modules are not absolutely acyclic.
  
  To prove these assertions, one can use the notions of coacyclic
  and contraacyclic complexes.  Given an abelian category
  $\mathcal{A}$ with exact coproducts, one says that a complex in
  $\mathcal{A}$ is \define{coacyclic} if it belongs to the
  minimal triangulated subcategory of $\K(\mathcal{A})$
  containing the absolutely acyclic complexes and closed under
  coproducts.  Dually, if an abelian category $\mathcal{A}$ has
  exact products, a complex in $\mathcal{A}$ is said to be
  \define{contraacyclic} if it belongs to the minimal
  triangulated subcategory of $\K(\mathcal{A})$ containing the
  absolutely acyclic complexes and closed under products.  The
  following lemmas summarize some key observations.
  
  \begin{lemma}
    \phantomsection\label{l:second-kind-orthogonality}
    \begin{enumerate}
    \item
      \label{item:KinjA:coacyclic=zero}
      Let $I\in\K(\inj(\mathcal{A}))\subset\K(\mathcal{A})$ be a
      complex of injective objects in an abelian category
      $\mathcal{A}$ with exact coproducts.  Then $I$ is coacyclic
      if and only if it is contractible.
    \item
      \label{item:KinjA:contraacyclic=zero}
      Let $P\in\K(\proj(\mathcal{A}))\subset\K(\mathcal{A})$ be a
      complex of projective objects in an abelian category
      $\mathcal{A}$ with exact products.  Then $P$ is
      contraacyclic if and only if it is contractible.
    \end{enumerate}
  \end{lemma}
  
  \begin{proof}
    Being contractible means being zero in
    $\K(\mathcal{A})$. Hence any contractible complex is
    coacyclic resp.\ contraacyclic.  The reverse implications
    follow
    from~\cite[Section~3.5]{positselski-two-kinds-of-derived}
    (or~\cite[Lemma~A.1.3]{positselski-contraherent-cosheaves}):
    In part~\ref{item:KinjA:coacyclic=zero}, we have
    $\Hom_{\K(\mathcal{A})}(A,I)=0$ for all coacyclic complexes
    $A$ in $\mathcal{A}$; 
    if $I$ is coacyclic, take $A=I$.
    In
    part~\ref{item:KinjA:contraacyclic=zero}, we have
    $\Hom_{\K(\mathcal{A})}(P,B)=0$ for all contraacyclic
    complexes $B$ in $\mathcal{A}$;
    if $P$ is contraacyclic, take $B=P$.
  \end{proof}

  \begin{lemma}
    \phantomsection\label{l:second-kind-bounded}
    \begin{enumerate}
    \item
      Any bounded below acyclic complex in $\mathcal{A}$ is
      coacyclic.
    \item
      Any bounded above acyclic complex in $\mathcal{A}$ is
      contraacyclic.
    \end{enumerate}
  \end{lemma}

  \begin{proof}
    This is a particular case of~\cite[Lemmas in sections~2.1
    and~4.1]{positselski-homological-semimodules} (see
    also~\cite[Theorem~1(a) in
    section~3.4]{positselski-two-kinds-of-derived}).
  \end{proof}

  Now we can show that the three acyclic
  complexes~\eqref{eq:unbounded}, \eqref{eq:bounded-above},
  \eqref{eq:bounded-below} are not absolutely acyclic. Convince
  yourself that, given any contractible complex $A$ in an abelian
  category, all monomorphisms $\B^n(A) \hra A^n$ and all
  epimorphisms $A^{n-1} \sra \B^n(A)$ split.  Since $M$ is not
  projective, $P_0 \xsra{\pi} M$ does not split (similarly,
  $M \xhra{\iota} J^0$ does not split since $M$ is not
  injective). Hence the complex \eqref{eq:unbounded} is not
  contractible. Since all its terms are both projective and
  injective, Lemma~\ref{l:second-kind-orthogonality} shows that
  it is neither coacyclic nor contraacyclic.

  Consider the obvious map
  $\eqref{eq:bounded-above} \ra \eqref{eq:unbounded}$ of
  complexes. Its mapping cone is clearly homotopic to
  \eqref{eq:bounded-below}. Hence there is a triangle
  $\eqref{eq:bounded-above} \ra \eqref{eq:unbounded} \ra
  \eqref{eq:bounded-below} \ra \Sigma \eqref{eq:bounded-above}$
  in $\K(\mathcal{A})$.  By
  Lemma~\ref{l:second-kind-bounded}(a), the
  complex~\eqref{eq:bounded-below} is coacyclic.  Since the
  complex~\eqref{eq:unbounded} is not coacyclic, the
  complex~\eqref{eq:bounded-above} cannot be coacyclic.  Dually,
  by Lemma~\ref{l:second-kind-bounded}(b), the
  complex~\eqref{eq:bounded-above} is contraacyclic.  Since the
  complex~\eqref{eq:unbounded} is not contraacyclic, the
  complex~\eqref{eq:bounded-below} cannot be contraacyclic.

  By definition, all absolutely acyclic complexes are both
  coacyclic and contraacyclic.  Thus none
  of the complexes~\eqref{eq:unbounded}, \eqref{eq:bounded-above},
  \eqref{eq:bounded-below} is absolutely acyclic.
\end{example}

Denote the thick subcategory of absolutely acyclic complexes by
$\Ac^\abs(\mathcal{A})\subset\K(\mathcal{A})$.  The triangulated
Verdier quotient category
\begin{equation*}
  \D^\abs(\mathcal{A})=\K(\mathcal{A})/\Ac^\abs(\mathcal{A})
\end{equation*}
is called the \define{absolute derived category} of the abelian
category $\mathcal{A}$.

Since all absolutely acyclic complexes are acyclic, we have a
canonical (triangulated Verdier quotient) functor
\begin{equation*}
  \D^\abs(\mathcal{A})\sra\D(\mathcal{A}).
\end{equation*}
In fact, we obtain a commutative triangular diagram of
triangulated Verdier quotient functors
$\K(\mathcal{A})\sra\D^\abs(\mathcal{A}) \sra\D(\mathcal{A})$.

\begin{remark}
  \label{r:gdim-finite-acyclic=absolute-acyclic}
  We have already mentioned that any bounded acyclic complex is
  absolutely acyclic.  In an abelian category of finite global
  dimension, any acyclic complex is absolutely acyclic
  \cite[Remark in
  section~2.1]{positselski-homological-semimodules}.  This means
  that the canonical functor
  $\D^\abs(\mathcal{A})\sra\D(\mathcal{A})$ is an equivalence
  whenever $\mathcal{A}$ has finite global dimension.
\end{remark}

\begin{theorem}
  \label{t:absolute-derived-fully-faithful}
  Let $\mathcal{A}$ be an abelian category and let
  $\mathcal{U}\subset \mathcal{A}$ be a non-empty, full
  subcategory closed under subobjects, quotients, and extensions.
  Assume that for any epimorphism $A\sra U$ in $\mathcal{A}$ with
  $U\in\mathcal{U}$ there exists a subobject $V\subset A$ with
  $V \in \mathcal{U}$ such that the composition $V \hra A \sra U$
  is an epimorphism.  Then the obvious triangulated functor
  \begin{equation*}
    \D^\abs(\mathcal{U})\ra\D^\abs(\mathcal{A})    
  \end{equation*}
  is fully faithful.
\end{theorem}

\begin{remark}
  We do not know how to describe the essential image of this
  functor
  in a nice way.
  In general, it is \emph{not} true that all complexes
  in $\mathcal{A}$ whose cohomology objects belong to
  $\mathcal{U}$ are in the essential image: Let $\mathcal{A}$ be
  an abelian category such that there are acyclic complexes that
  are not absolutely acyclic (cf.\
  Example~\ref{ex:acyclic-not-absolutely}) and consider the
  subcategory $\mathcal{U}=0$.
\end{remark}

\begin{proof}
  We introduce some additional terminology in order to formulate
  the argument.  First of all, the objects of
  $\Ac^\abs(\mathcal{U})$ will be called \define{absolutely
    acyclic complexes with respect to $\mathcal{U}$}, the objects
  of $\Ac^\abs(\mathcal{A})$ will be called \define{absolutely
    acyclic complexes with respect to $\mathcal{A}$}.  Clearly,
  any absolutely acyclic complex with respect to $\mathcal{U}$ is
  also absolutely acyclic with respect to $\mathcal{A}$.
  Conversely, it follows from the assertion of the theorem that a
  complex in $\mathcal{U}$ that is absolutely acyclic with
  respect to $\mathcal{A}$ is already absolutely acyclic with
  respect to $\mathcal{U}$.  So
  $\Ac^\abs(\mathcal{U}) =
  \Ac^\abs(\mathcal{A})\cap\K(\mathcal{U})$.  But this is not
  obvious and needs to be proved.
  
  We will show that any morphism in $\K(\mathcal{A})$ from a
  complex in $\mathcal{U}$ to a complex absolutely acyclic with
  respect to $\mathcal{A}$ factors through a complex absolutely
  acyclic with respect to $\mathcal{U}$.  This is clearly
  sufficient to prove the theorem (see e.\,g.\
  \cite[Proposition~B.2(II)]{lunts-schnuerer-mf-sod}).

  The full subcategory $\Ac^\abs(\mathcal{A})$ of absolutely
  acyclic complexes in $\K(\mathcal{A})$ is defined by a
  generation procedure, and the argument proceeds along the steps
  of this procedure.  Let us describe this procedure in more
  detail.  We start with the empty subcategory of
  $\K(\mathcal{A})$ and successively add objects.
  \begin{enumerate}[label=Step \arabic*.]
  \item All total complexes $\Tot(K\ra L\ra M)$ of short
    exact sequences $0\ra K\ra L\ra M\ra 0$ of complexes in
    $\mathcal{A}$ are added.
  \item If two complexes $A$ and $B$ have been added in the
    previous steps, and if $t\colon A\ra B$ is a morphism of
    complexes, then we add the complex $C=\Cone(t)$.  Let us
    emphasize that $\Cone(t)$ is understood here as the cone of a
    closed morphism of complexes; so there exists a short exact
    sequence $0\ra B\ra C\ra\Sigma(A)\ra0$ in the abelian
    category of complexes in $\mathcal{A}$.  All the complexes
    $C$ that can be obtained in this way are added in this step.
  \item Step 2 is repeated countably many times.
  \item If a complex $A$ has been added already and a complex $B$
    is homotopy equivalent to $A$, then the complex $B$ is added
    in this step.
  \item If a complex $A$ has been added already and a complex $B$
    is a direct summand of $A$ in $\K(\mathcal{A})$, then the
    complex $B$ is added in this step.
  \end{enumerate}
  After all the steps 1--5 have been run, we have precisely
  produced the full subcategory
  $\Ac^\abs(\mathcal{A})\subset \K(\mathcal{A})$.

  Let us say that a complex $A\in\K(\mathcal{A})$ is
  \define{approachable} if any morphism $f\colon U\ra A$ in
  $\K(\mathcal{A})$ with $U\in\K(\mathcal{U})$ factors through a
  complex absolutely acyclic with respect to $\mathcal{U}$.
  Moreover, we will say that a complex $A$ is \define{strongly
    approachable} if, for any subcomplex $D\subset A$ with terms
  $D^n\in\mathcal{U}$, $n\in\bZ$, there exists a subcomplex
  $E\subset A$ such that $D\subset E$ and the complex $E$ is
  absolutely acyclic with respect to $\mathcal{U}$.
  Since $\mathcal{U}$ is closed under quotients,
  any
  strongly approachable complex is approachable.  Our aim is to
  show that all complexes $A\in\Ac^\abs(\mathcal{A})$ are
  approachable.

  Approachability is a concept defined on the level of
  homotopy categories.
  Strong approachability is a concept defined
  on the level of
  abelian categories of complexes.
  In the
  above generation procedure, steps 1--3 take place on the
  abelian level.  Steps 4--5 take place on the level of the
  homotopy category.

  The proof of the theorem now proceeds as a sequence of lemmas.

  \begin{lemma}
    \label{l:totalizations}
    For any short exact sequence $0\ra K\ra L\ra M\ra0$ of
    complexes in $\mathcal{A}$, the complex $A=\Tot(K\ra L\ra M)$
    is strongly approachable.
  \end{lemma}

  \begin{proof}
    We have $A^n=K^{n+1}\oplus L^n\oplus M^{n-1}$ for all
    $n\in\bZ$.  Let $D\subset A$ be a subcomplex with
    terms
    $D^n\in\mathcal{U}$.  Denote by $D_K^{n+1}\subset K^{n+1}$,
    $D_L^n\subset L^n$, and $D_M^{n-1}\subset M^{n-1}$ the images
    of the projections of the subobject $D^n\subset A^n$ onto the
    three direct summands of the object $A^n$.  So we have
    $D^n\subset D_K^{n+1}\oplus D_L^n\oplus D_M^{n-1}$ and all
    three objects $D_K^{n+1}$, $D_L^n$, $D_M^{n-1}$ belong to
    $\mathcal{U}$.

    By assumption, for each $n\in\bZ$ we can choose a subobject
    $G^n\subset L^n$ such that $G^n\in\mathcal{U}$ and such that
    the image of $G^n$ under the epimorphism $L^n\sra M^n$
    coincides with $D_M^n\subset M^n$.  Denote by
    $H^n\subset L^n$ the image of the composition of
    monomorphisms $D_K^n\hra K^n\hra L^n$.  Put
    $Q^n=H^n+D_L^n+G^n\subset L^n$ and
    $V^n=Q^n+d(Q^{n-1})\subset L^n$.  Clearly, we have
    $Q^n\in\mathcal{U}$ and consequently $V^n\in \mathcal{U}$.

    Now $V$ is a subcomplex of $L$.  Denote by $U$ the preimage
    of $V$ under the monomorphism of complexes $K\hra L$, and by
    $W$ the image of $V$ under the epimorphism of complexes
    $L\sra M$.  Then we have a short exact sequence
    $0\ra U\ra V\ra W\ra0$ of complexes in $\mathcal{U}$.  The
    complex $E=\Tot(U\ra V\ra W)$ is absolutely acyclic with
    respect to $\mathcal{U}$, and it is a subcomplex in $A$.
    Finally, we have $D^n_K\subset U^n\subset K^n$,
    $D^n_L\subset V^n \subset L^n$, and
    $D^n_M\subset W^n\subset M^n$, hence $D\subset E\subset A$,
    as desired.
  \end{proof}

  \begin{lemma}
    \label{l:cones}
    If $t\colon A\ra B$ is a morphism of complexes in
    $\mathcal{A}$ and the two complexes $A$ and $B$ are strongly
    approachable, then the complex $C=\Cone(t)$ is strongly
    approachable.
  \end{lemma}

  \begin{proof}
    The terms of the complex $C$ are $C^n=A^{n+1}\oplus B^n$.
    Let $D\subset C$ be a subcomplex with terms
    $D^n\in\mathcal{U}$.  Denote by $D_A^{n+1}\subset A^{n+1}$
    and $D_B^n\subset B^n$ the images of the projections of the
    subobject $D^n\subset C^n$ onto the two direct summands of
    $C^n$.  So we have $D^n\subset D^{n+1}_A\oplus D^n_B$ and
    $D_A^{n+1}$, $D_B^n\in\mathcal{U}$.

    Put $P^n=D_A^n+d(D_A^{n-1})\subset A^n$.  Then $P$ is a
    subcomplex of $A$ with terms $P^n\in\mathcal{U}$.  By
    assumption, there exists a subcomplex $E\subset A$ such that
    $E$ is absolutely acyclic with respect to $\mathcal{U}$
    and $P\subset E\subset A$.  In particular, we have
    $E^n\in\mathcal{U}$ for all $n\in\bZ$.  Put
    $Q^n=D_B^n+d(D_B^{n-1})+t(E^n)\subset B^n$.  Then $Q$ is a
    subcomplex of $B$ with terms $Q^n\in\mathcal{U}$.  By
    assumption, there exists a subcomplex $F\subset B$ such that
    $F$ is absolutely acyclic with respect to $\mathcal{U}$ and
    $Q\subset F\subset B$.

    By construction, we have $t(E)\subset F$; denote by
    $h\colon E\ra F$ the morphism induced by $t$.  Put
    $G=\Cone(h)$.  Then the complex $G$ is absolutely acyclic
    with respect to $\mathcal{U}$, and it is a subcomplex of $C$.
    Finally, we have $D_A^n\subset E^n\subset A^n$ and
    $D_B^n\subset F^n \subset B^n$, hence $D\subset G\subset C$,
    as desired.
  \end{proof}

  \begin{lemma}
    \label{l:homotopy}
    \begin{enumerate}
    \item If a complex $A$ is approachable and a complex $B$ is
      homotopy equivalent to $A$, then the complex $B$ is
      approachable as well.
    \item If a complex $A\in\K(\mathcal{A})$ is approachable and
      a complex $B$ is a direct summand of $A$ in
      $\K(\mathcal{A})$, then the complex $B$ is approachable as
      well.
    \end{enumerate}
  \end{lemma}

  \begin{proof}
    This follows immediately from the definitions.
  \end{proof}

  It remains to finish the proof of the theorem.  All the
  complexes produced in step 1 are strongly approachable by
  Lemma~\ref{l:totalizations}.  Hence all the complexes produced
  in steps 2 and 3 are strongly approachable by
  Lemma~\ref{l:cones}.  It follows that all complexes produced in
  steps 4 and 5 are approachable by Lemma~\ref{l:homotopy}.  Thus
  all complexes that are absolutely acyclic with respect to
  $\mathcal{A}$ are approachable.
\end{proof}

\begin{theorem}
  \label{t:serre-subcat-gdim-finite}
  Let $\mathcal{A}$ be an abelian category of finite global
  dimension with exact countable coproducts, and let
  $\mathcal{U}\subset\mathcal{A}$ be a non-empty, full
  subcategory closed under subobjects, quotients, and extensions.
  Assume that for any epimorphism $A\sra U$ in $\mathcal{A}$ with
  $U\in\mathcal{U}$ there exists a subobject $V\subset A$ with
  $V\in \mathcal{U}$ such that the composition $V\hra A\sra U$ is
  an epimorphism.  Then the obvious triangulated functor is an
  equivalence
  $\D(\mathcal{U}) \sira \D_{\mathcal{U}}(\mathcal{A})$.
\end{theorem}

\begin{proof}
  Proposition~\ref{p:D-minus-equivalence} shows that the functor
  $\D^-(\mathcal{U})\ra\D^-_{\mathcal{U}}(\mathcal{A})$ is an
  equivalence of triangulated categories.  In particular, the
  inclusion of abelian categories $\mathcal{U}\ra \mathcal{A}$
  induces isomorphisms on all Ext groups.  Therefore, the global
  dimension of $\mathcal{U}$ cannot exceed the global dimension
  of $\mathcal{A}$.

  Since the global dimensions of the abelian categories
  $\mathcal{U}$ and $\mathcal{A}$ are finite, both vertical
  functors in the obvious commutative diagram
  \begin{equation*}
    \xymatrix{
      {\D^\abs(\mathcal{U})} \ar[r] \ar[d]^-{\sim} & 
      {\D^\abs(\mathcal{A})} \ar[d]^-{\sim} \\
      {\D(\mathcal{U})} \ar[r] & 
      {\D(\mathcal{A})}
    }
  \end{equation*}
  of triangulated functors are equivalences by
  Remark~\ref{r:gdim-finite-acyclic=absolute-acyclic}.  The upper
  horizontal functor is fully faithful by
  Theorem~\ref{t:absolute-derived-fully-faithful}.  It follows
  that the lower horizontal functor is fully faithful.  Finally,
  the functor $\D(\mathcal{U})\ra \D_{\mathcal{U}}(\mathcal{A})$
  is essentially surjective by
  Proposition~\ref{p:essentially-surjective}.
\end{proof}

The
following
corollary concerns locally noetherian Grothendieck
abelian categories
(see~\cite[Section~II.4]{gabriel-categories-abeliennes}).

\begin{corollary}
  \label{c:locally-noetherian-category}
  Let $\mathcal{A}$ be a locally noetherian Grothendieck category
  of finite global dimension and let
  $\mathcal{U}\subset\mathcal{A}$ be the full subcategory of
  noetherian objects.  Then the obvious triangulated functor is
  an equivalence
  $\D(\mathcal{U}) \sira \D_{\mathcal U}(\mathcal{A})$.
\end{corollary}

\begin{proof}
  This is a particular case of
  Theorem~\ref{t:serre-subcat-gdim-finite} because the categories
  $\mathcal{A}$ and $\mathcal{U}$ clearly satisfy the necessary
  assumptions.
\end{proof}

For a noetherian scheme $X$, we denote by $\Qcoh(X)$ the abelian
category of quasi-coherent sheaves on $X$ and by
$\coh(X)\subset \Qcoh(X)$ the full subcategory of coherent
sheaves.  The notation $\D_{\coh}(\Qcoh(X))$ stands for the full
triangulated subcategory of $\D(\Qcoh(X))$ consisting of all
complexes with coherent cohomology sheaves.

\begin{corollary}
  \label{c:regular-noetherian-finite-dim}
  Let $X$ be a regular noetherian scheme of finite Krull
  dimension.  Then the obvious triangulated functor is an
  equivalence
  \begin{equation*}
    \D(\coh(X)) \sira \ D_{\coh}(\Qcoh(X)).    
  \end{equation*}
\end{corollary}

\begin{proof}
  For any noetherian scheme $X$, the category $\Qcoh(X)$ is
  locally noetherian
  \cite[Th\'eor\`eme~VI.2.1]{gabriel-categories-abeliennes} and
  its subcategory of noetherian objects is $\coh(X)$.  When $X$
  is regular of Krull dimension~$d$, the global dimension of
  $\Qcoh(X)$ is equal to $d$.  This well-known result is provable
  as follows.

  Assume first that $X=\Spec R$ is affine. Let $M$ be a finitely
  generated $R$-module. Let
  $0 \ra Z \ra P_{d-1} \ra \dots \ra P_0 \sra M \ra 0$ be an exact
  sequence where all $P_i$ are finitely generated projective
  $R$-modules. If we localize this sequence at an arbitrary prime
  ideal $\mfp \in \Spec R$, the $R_\mfp$-module $Z_\mfp$ is
  projective because the regular noetherian local ring $R_\mfp$
  has global dimension equal to $\dim R_\mfp \leq \dim R=d$ by
  \cite[Theorem~19.2]{matsumura-comm-ring}.  Hence $Z$ is
  projective as an $R$-module. This shows that
  $\mod(R) \cong \coh(\Spec R)$ has global dimension $\leq d$.
  But this implies that $\Mod(R) \cong \Qcoh(\Spec R)$ has global
  dimension $\leq d$ and in fact equal to $d$ by
  \cite[Theorem~8.16, Theorem~8.55]{rotman}.

  Now assume that $X$ is not necessarily affine.  By
  \cite[Theorem~II.7.18]{hartshorne-residues-duality}, any
  $\mathcal{F} \in \Qcoh(X)$ admits an embedding
  $\mathcal{F} \subset \mathcal{I}$ where $\mathcal{I}$ is
  quasi-coherent and injective as an $\mathcal{O}_X$-module and
  hence in particular injective as an object of $\Qcoh(X)$. Since
  this embedding splits if $\mathcal{F}$ is injective in
  $\Qcoh(X)$, we obtain
  $\inj(\Qcoh(X))=\inj(\Mod(\mathcal{O}_X)) \cap \Qcoh(X)$.  From
  \cite[Lemma~II.7.16]{hartshorne-residues-duality} we see that
  an object of $\Qcoh(X)$ is injective if and only if its
  restrictions to an arbitrary open cover of $X$ are injective as
  quasi-coherent
  sheaves.

  Given $\mathcal{F} \in \Qcoh(X)$, let
  $0 \ra \mathcal{F} \ra \mathcal{I}_0 \ra \dots \ra
  \mathcal{I}_{d-1} \ra \mathcal{C} \ra 0$ be an exact sequence
  in $\Qcoh(X)$ where all $\mathcal{I}_i$ are injective. If we
  restrict this sequence to an arbitrary affine open subset $U$
  of $X$, the affine case treated above shows that
  $\mathcal{C}|_U$ is injective in $\Qcoh(U)$. Hence
  $\mathcal{C}$ is injective in $\Qcoh(X)$. This shows that the
  global dimension of $\Qcoh(X)$ is $\leq d$. Similarly as above,
  it must in fact be equal to $d$.
\end{proof}

\begin{remark}
  \label{r:gdim-locally-noetherian}
  Let $\mathcal{A}$ be a locally noetherian Grothendieck category
  and let $\mathcal{U}\subset\mathcal{A}$ be the full subcategory
  of noetherian objects.  Then the global dimensions of the
  abelian categories $\mathcal{U}$ and $\mathcal{A}$ coincide.
  Indeed, the argument of the first paragraph of the proof of
  Theorem~\ref{t:serre-subcat-gdim-finite} shows that the global
  dimension of $\mathcal{U}$ cannot exceed the global dimension
  of $\mathcal{A}$.  The converse inequality is provable using
  the next two lemmas.
\end{remark}

\begin{lemma}
  \label{l:loc-noetherian-baer}
  An object $J\in\mathcal{A}$ is injective if and only if, for
  every object $U\in\mathcal{U}$ and every subobject $V\hra U$,
  any morphism $V\to J$ can be extended to a morphism $U\to J$.
\end{lemma}

\begin{proof}
  The non-trivial implication
  is a standard Zorn lemma argument very similar to the
  proof of the Baer criterion for injectivity of modules.
  Given an object $A\in\mathcal{A}$, a subobject $B\hra A$, and
  a morphism $B\to J$, one considers the poset of all morphisms
  $C\to J$ defined on intermediate subobjects $B\hra C\hra A$
  and making the triangular diagram $B\hra C\to J$ commutative.
  This poset clearly has unions of chains, and therefore it has
  a maximal element $C\to J$.  We claim that $C=A$. If not,
  there is a noetherian subobject $U\hra A$ that is not
  contained in $C$. The restriction of the morphism $C\to J$ to
  the subobject $V=U\cap C\hra C$ has by assumption an
  extension along the inclusion $V\hra U$ to a morphism
  $U \to J$.  The two morphisms $C\to J$ and $U\to J$ agree on
  the intersection $U\cap C$ and can therefore be uniquely
  extended to a morphism $C+U\to J$. Since $C \subsetneq C+U$
  we obtain a contradiction.
\end{proof}

\begin{lemma}
  \label{l:ext-filtered-colimit}
  For any object $U\in\mathcal{U}$ and every integer $n\ge0$,
  the functor $\Ext^n_{\mathcal{A}}(U,{-})$ takes filtered
  colimits in $\mathcal{A}$ to filtered colimits of abelian
  groups.
\end{lemma}

\begin{proof}
  Given an object $A\in\mathcal{A}$, the abelian group
  $\Ext^n_{\mathcal{A}}(U,A)=\Hom_{\D(\mathcal{A})}(U,\Sigma^n
  A)$ can be computed as the set of all equivalence classes of
  roofs $U\leftarrow P\rightarrow \Sigma^nA$, where $P\to U$ is
  a quasi-isomorphism of complexes in $\mathcal{A}$.  Clearly,
  it suffices to consider bounded above complexes $P$.
  Following the proof of
  Proposition~\ref{p:D-minus-equivalence}, we can assume that
  $P$ is a complex in $\mathcal{U}$.  Recall from \cite[II.4,
  Corollaire~1]{gabriel-categories-abeliennes} that, given any
  object $V \in \mathcal{U}$, the functor
  $\Hom_\mathcal{A}(V, -)$ takes filtered colimits in
  $\mathcal{A}$ to filtered colimits of abelian groups.
  If now $A=\colim_{i \in I} A_i$ is a
  filtered colimit in $\mathcal{A}$, it is straightforward
  to see
  that the natural map
  $\colim_{i \in I} \Hom_{\K(\mathcal{A})}(P,\Sigma^n A_i) \ra
  \Hom_{\K(\mathcal{A})}(P,\Sigma^n A)$ is bijective, for all
  complexes $P$ in $\mathcal{U}$; we deduce that 
  the canonical map
  $\colim_{i \in I} \Hom_{\D(\mathcal{A})}(U, \Sigma^n A_i)
  \ra \Hom_{\D(\mathcal{A})}(U, A)$ is bijective as well.
  
  Alternatively, one can devise an argument based on the fact
  that the class of injective objects in
  any
  locally noetherian
  Grothendieck
  category is closed under filtered colimits (see \cite[II.4,
  Corollaire~1]{gabriel-categories-abeliennes}).  Given a
  diagram $(A_i)_{i\in I}$ in $\mathcal{A}$ indexed by a
  directed poset $I$, one then needs to construct an
  $I$-indexed diagram of injective resolutions $A_i\to J_i$ of
  the objects $A_i$.  To this end, one can observe that
  functorial injective resolutions exist in any Grothendieck
  category (see e.\,g.\
  \cite[\sptag{079H}]{stacks-project}).  One can also consider
  the Grothendieck category $\mathcal{A}^I$ of $I$-indexed
  diagrams in $\mathcal{A}$ and notice that any injective
  object in $\mathcal{A}^I$ is a diagram of injective objects
  in $\mathcal{A}$, because the functor
  $\mathcal{A}^I\to\mathcal{A}$ taking a diagram
  $(K_i)_{i\in I}$ to the object $K_j$ for a fixed index
  $j\in I$ is right adjoint to an exact functor
  $\mathcal{A} \to\mathcal{A}^I$ and consequently takes
  injectives to injectives.  So any injective resolution of
  $(A_i)_{i\in I}$ in $\mathcal{A}^I$ provides the desired
  diagram of injective resolutions.

  Then $\colim_{i\in I} J_i$ is an injective resolution of the
  object $\colim_{i\in I} A_i$ in $\mathcal{A}$.  Computing
  $\Ext^n_{\mathcal{A}}(U,A_i)$ in terms of the injective
  resolution $J_i$ of the object $A_i$ and
  $\Ext^n_{\mathcal{A}}(U,\colim_{i\in I}A_i)$ in terms of
  the injective resolution $\colim_{i\in I} J_i$ of the object
  $\colim_{i\in I}A_i$, and using the fact that the functor
  $\Hom_{\mathcal{A}}(U,{-})$ takes filtered colimits in
  $\mathcal{A}$ to filtered colimits of abelian groups (see
  \cite[II.4, Corollaire~1]{gabriel-categories-abeliennes}),
  one deduces the same property for the functor
  $\Ext^n_{\mathcal{A}}(U,{-})$.
\end{proof}

Now we can finish the proof of the converse inequality.  Let
$d$ be the global dimension of $\mathcal{U}$.  If $d=\infty$,
then there is nothing to prove.  The case $d=-1$ is also
trivial.  So we assume that $d\ge0$ is a natural number.  Then
we have $\Ext_{\mathcal{A}}^{d+1}(U,W)=0$ for all $U$,
$W\in\mathcal{U}$.  Let $A \in \mathcal{A}$ be arbitrary.
Since $A$ is the filtered colimit of its noetherian subobjects,
Lemma~\ref{l:ext-filtered-colimit} shows that
$\Ext_{\mathcal{A}}^{d+1}(U,A)=0$ for all $U\in\mathcal{U}$.
Let $J$ be an injective resolution of $A$ in $\mathcal{A}$, and
let $Z^d$ be the kernel of the differential $J^d\to J^{d+1}$.
Then
$\Ext^1_{\mathcal{A}}(U,Z^d)=\Ext^{d+1}_{\mathcal{A}} (U,A)=0$
for all $U\in\mathcal{U}$.  By
Lemma~\ref{l:loc-noetherian-baer}, we conclude that the object
$Z^d\in\mathcal{A}$ is injective.  This implies that the global
dimension of $\mathcal{A}$ is $\leq d$.

\appendix

\section{The hereditary case}
\label{sec:hereditary-case}

The aim of this appendix is to give a full proof of
Corollary~\ref{c:equiv-hereditary}.
This is an easy consequence of
  Theorem~\ref{t:D-hereditary}, a structural result on the
  unbounded derived category of a hereditary abelian category.

If $A$ and $B$ are objects of an abelian category $\mathcal{A}$,
we may view them as complexes concentrated in degree zero and
then have a canonical identification
$\Hom_\mathcal{A}(A,B)=\Hom_{\D(\mathcal{A})}(A,B)$.
By definition, we have
$\Ext^n_\mathcal{A}(A,B)=\Hom_{\D(\mathcal{A})}(A,\Sigma^n B)$
for $n \in \bZ$. We remind the reader that an abelian category is
called \define{hereditary} if and only if
$\Ext^2_\mathcal{A}(A,B)=0$ for all objects $A$, $B$ of
$\mathcal{A}$.

\begin{theorem}
  [{cf.\ \cite[Section 1.6]{krause-chicago}}]
  \label{t:D-hereditary}
  Let $\mathcal{A}$ be a hereditary abelian category.  Then any
  object $M$ in $\D(\mathcal{A})$ is isomorphic to a complex with
  vanishing differentials whose $n$-th component is then
  necessarily isomorphic to
  $\HH^n(M)$.
  
  Let $X=((X^n), 0)$ and $Y=((Y^n),0)$ be complexes in
  $\mathcal{A}$ with vanishing differentials. Let
  $\iota_n \colon \Sigma^{-n}X^n \ra X$ and
  $\pi_n \colon Y \ra \Sigma^{-n}Y^n$ be the obvious inclusion
  and projection morphisms of complexes.  Then the map
  \begin{align}
    \label{eq:Hom-D-hereditary}
    \Hom_{\D(\mathcal{A})}(X,Y)
    & \sira
      \prod_{n \in \bZ} \Hom_\mathcal{A}(X^n, Y^n)
      \times
      \prod_{n \in \bZ} \Ext^1_\mathcal{A}(X^n, Y^{n-1}),\\
    \notag
    f & \mapsto \big((\Sigma^n(\pi_n \circ f \circ \iota_n))_{n \in
        \bZ}, 
        (\Sigma^n(\pi_{n-1} \circ f \circ \iota_n))_{n \in \bZ}\big),  
  \end{align}
  is bijective (note that
  $\Sigma^n(\pi_n \circ f \circ \iota_n)=\HH^n(f)$).

  In particular, any complex $Z=((Z^n), 0)$ in $\mathcal{A}$ with
  vanishing differentials is both coproduct and product in
  $\D(\mathcal{A})$ of the objects $\Sigma^{-n} Z^n$, for
  $n \in \bZ$, with obvious inclusion and projection morphisms.
  
  Altogether we have
  $\bigoplus_{n \in \bZ} \Sigma^{-n} \HH^n(M) \cong M \cong
  \prod_{n \in \bZ} \Sigma^{-n}\HH^n(M)$ for any object $M$ of
  $\D(\mathcal{A})$.
\end{theorem}

\begin{proof}
  The first part of the proof is copied from \cite[Section
  1.6]{krause-chicago}. The second part, bijectivity of
  \eqref{eq:Hom-D-hereditary}, is treated rather quickly in
  loc.\,cit.; we therefore provide full details.
  
  Let $M$ be a complex in $\mathcal{A}$.  For $n \in \bZ$, the
  epimorphism $d'=d'^{n-1} \colon M^{n-1} \sra \B^n(M)$ induced
  by $d=d^{n-1} \colon M^{n-1} \ra M^n$ yields a surjection
  $\Ext^1_\mathcal{A}(\HH^n(M), M^{n-1}) \sra
  \Ext^1_\mathcal{A}(\HH^n(M), \B^n(M))$ because
  $\Ext^2_\mathcal{A}$ vanishes.  Hence the extension
  $0 \ra \B^n(M) \xra{i=i^n} \Z^n(M) \xra{p=p^n} \HH^n(M) \ra 0$
  comes from an extension
  $0 \ra M^{n-1} \xra{j=j^n} E^n \xra{q=q^n} \HH^n(M) \ra 0$,
  i.\,e.\ there is a commutative diagram
  \begin{equation*}
    \xymatrix{
      {0} \ar[r] & {M^{n-1}} \ar[r]^-{j} \ar[d]_-{d'}
      & {E^n} \ar[r]^-{q} \ar[d]_-{u'=u'^n} & {\HH^n(M)} \ar[r]
      \ar[d]_-{\id} 
      & {0} 
      \\ 
      {0} \ar[r] & {\B^n(M)} \ar[r]^-{i} & {\Z^n(M)} \ar[r]^-{p} &
      {\HH^n(M)}
      \ar[r] & 0 
    }
  \end{equation*}
  in $\mathcal{A}$ whose left square is cocartesian.  It is easy
  to check that the diagram
  \begin{equation}
    \label{eq:M-qiso-HM}
    \xymatrix{
      {\dots} \ar[r]^-{0} &
      {\HH^{n-1}(M)} \ar[r]^-{0} &
      {\HH^n(M)} \ar[r]^-{0} &
      {\HH^{n+1}(M)} \ar[r]^-{0} &
      {\dots}
      \\
      {\dots}
      \ar[r]^-{
        \left(
          \begin{smallmatrix}
            0 & j\\ 0 & 0
          \end{smallmatrix}
        \right)
      }
      &
      {E^{n-1} \oplus M^{n-1}}
      \ar[u]^-{(q,0)}
      \ar[r]^-{
        \left(
          \begin{smallmatrix}
            0 & j\\ 0 & 0
          \end{smallmatrix}
        \right)
      }
      \ar[d]_-{(u,\id)} &
      {E^n \oplus M^n}
      \ar[u]^-{(q,0)}
      \ar[r]^-{
        \left(
          \begin{smallmatrix}
            0 & j\\ 0 & 0
          \end{smallmatrix}
        \right)
      }
      \ar[d]_-{(u,\id)} &
      {E^{n+1} \oplus M^{n+1}}
      \ar[u]^-{(q,0)}
      \ar[r]^-{
        \left(
          \begin{smallmatrix}
            0 & j\\ 0 & 0
          \end{smallmatrix}
        \right)
      }
      \ar[d]_-{(u,\id)} &
      {\dots}
      \\
      {\dots} \ar[r]^-{d} &
      {M^{n-1}} \ar[r]^-{d} &
      {M^n} \ar[r]^-{d} &
      {M^{n+1}} \ar[r]^-{d} &
      {\dots}
    }
  \end{equation}
  in $\mathcal{A}$ is commutative where $u=u^n$ is the
  composition $E^n \xra{u'} \Z^n(M) \hra M^n$. The rows are
  complexes and the vertical arrows define quasi-isomorphisms
  between these complexes. This proves the first claim.

  We now prove surjectivity of \eqref{eq:Hom-D-hereditary}.
  Assume that morphisms $g_n \colon X^n \ra Y^n$ in $\mathcal{A}$
  and $e_n \colon X^n \ra \Sigma Y^{n-1}$ in $\D(\mathcal{A})$
  are given, for all $n \in \bZ$.  Represent $e_n$ by a short
  exact sequence
  $0 \ra Y^{n-1} \xra{k=k^n} F^n \xra{r=r^n} X^n \ra 0$ and
  consider the commutative diagram
  \begin{equation*}
    \xymatrix{
      {\dots} \ar[r]^-{0} &
      {X^{n-1}} \ar[r]^-{0} &
      {X^n} \ar[r]^-{0} &
      {X^{n+1}} \ar[r]^-{0} &
      {\dots}
      \\
      {\dots}
      \ar[r]^-{
        \left(
          \begin{smallmatrix}
            0 & k\\ 0 & 0
          \end{smallmatrix}
        \right)
      }
      &
      {F^{n-1} \oplus Y^{n-1}}
      \ar[u]^-{(r,0)}
      \ar[r]^-{
        \left(
          \begin{smallmatrix}
            0 & k\\ 0 & 0
          \end{smallmatrix}
        \right)
      }
      \ar[d]_-{(g \circ r,\id)} &
      {F^n \oplus Y^n}
      \ar[u]^-{(r,0)}
      \ar[r]^-{
        \left(
          \begin{smallmatrix}
            0 & k\\ 0 & 0
          \end{smallmatrix}
        \right)
      }
      \ar[d]_-{(g \circ r,\id)} &
      {F^{n+1} \oplus Y^{n+1}}
      \ar[u]^-{(r,0)}
      \ar[r]^-{
        \left(
          \begin{smallmatrix}
            0 & k\\ 0 & 0
          \end{smallmatrix}
        \right)
      }
      \ar[d]_-{(g \circ r,\id)} &
      {\dots}
      \\
      {\dots} \ar[r]^-{0} &
      {Y^{n-1}} \ar[r]^-{0} &
      {Y^n} \ar[r]^-{0} &
      {Y^{n+1}} \ar[r]^-{0} &
      {\dots}
    }
  \end{equation*}
  whose rows are complexes.  The vertical arrows define morphisms
  of complexes, and the upper morphism $(r,0)$ is a
  quasi-isomorphism.  We interprete this diagram as a roof which
  gives rise to a morphism $f \colon X \ra Y$ in
  $\D(\mathcal{A})$. This morphism $f$ is mapped to the element
  $((g_n), (e_n))$; to see this, use the subcomplex
  $\dots \ra 0 \ra Y^{n-1} \xra{k} F^n \ra 0 \ra \dots$ of the
  middle row together with its quasi-isomorphism to the subcomplex
  $\Sigma^{-n} X^n$ of the upper row.

  We now prepare for the proof of injectivity of
  \eqref{eq:Hom-D-hereditary}.  Let $f \colon X \ra Y$ be any
  morphism in $\D(\mathcal{A})$.  Represent $f$ by a roof
  $X \xla{a} M \xra{b} Y$ where $M$ is a suitable complex in
  $\mathcal{A}$, both $a$ and $b$ are morphisms of complexes and
  $a$ is a quasi-isomorphism. We deal with $M$ as explained in
  the first part of this proof.  By precomposing with the
  quasi-isomorphism $(u, \id)$ between middle and lower row in
  \eqref{eq:M-qiso-HM} we obtain the representation of $f$ by the
  roof
  \begin{equation}
    \label{eq:roof}
    \xymatrix{
      {\dots} \ar[r]^-{0} &
      {X^{n-1}} \ar[r]^-{0} &
      {X^n} \ar[r]^-{0} &
      {X^{n+1}} \ar[r]^-{0} &
      {\dots}
      \\
      {\dots}
      \ar[r]^-{
        \left(
          \begin{smallmatrix}
            0 & j\\ 0 & 0
          \end{smallmatrix}
        \right)
      }
      &
      {E^{n-1} \oplus M^{n-1}}
      \ar[u]^-{(a \circ u, a)}
      \ar[r]^-{
        \left(
          \begin{smallmatrix}
            0 & j\\ 0 & 0
          \end{smallmatrix}
        \right)
      }
      \ar[d]_-{(b \circ u,b)} &
      {E^n \oplus M^n}
      \ar[u]^-{(a \circ u, a)}
      \ar[r]^-{
        \left(
          \begin{smallmatrix}
            0 & j\\ 0 & 0
          \end{smallmatrix}
        \right)
      }
      \ar[d]_-{(b \circ u,b)} &
      {E^{n+1} \oplus M^{n+1}}
      \ar[u]^-{(a \circ u, a)}
      \ar[r]^-{
        \left(
          \begin{smallmatrix}
            0 & j\\ 0 & 0
          \end{smallmatrix}
        \right)
      }
      \ar[d]_-{(b \circ u,b)} &
      {\dots}
      \\
      {\dots} \ar[r]^-{0} &
      {Y^{n-1}} \ar[r]^-{0} &
      {Y^n} \ar[r]^-{0} &
      {Y^{n+1}} \ar[r]^-{0} &
      {\dots.}
    }
  \end{equation}
  Commutativity of the upper squares shows
  $a \circ u \circ j=a^n \circ u^n \circ j^n=0$ and hence
  $a \circ u=\ol{a} \circ q$ for a unique morphism
  $\ol{a}=\ol{a}^n \colon \HH^n(M) \ra X^n$. Since $a$ is a
  quasi-isomorphism, $\ol{a}=\ol{a}^n$ is an isomorphism.
  Commutativity of the lower squares shows that
  $b \circ u \circ j=b^n \circ u^n \circ j^n=0$ and hence
  $b \circ u=\ol{b} \circ q$ for a unique morphism
  $\ol{b}=\ol{b}^n \colon \HH^n(M) \ra Y^n$.

  Using the subcomplex
  $\dots \ra 0 \ra M^{n-1} \xra{j} E^n \ra 0 \ra \dots$ of the
  middle row of \eqref{eq:roof} together with its
  quasi-isomorphism to $\Sigma^{-n}X^n$ we first see that
  \begin{equation*}
    \Sigma^n(\pi_n \circ f \circ \iota_n)=
    \ol{b} \circ \ol{a}^{-1} \colon X^n
    \xsira{\ol{a}^{-1}}
    \HH^n(M) \xra{\ol{b}} Y^n.
  \end{equation*}
  Second we see that $\pi_{n-1} \circ f \circ \iota_n$ is given
  by the roof
  \begin{equation*}
    \xymatrix{
      {\dots} \ar[r]^-{0} &
      {0} \ar[r]^-{0} &
      {0} \ar[r]^-{0} &
      {X^n} \ar[r]^-{0} &
      {0} \ar[r]^-{0} &
      {\dots}
      \\
      {\dots}
      \ar[r]^-{0}
      &
      {0} \ar[r]^-{0} &
      {M^{n-1}}
      \ar[u]_-{0}
      \ar[r]^-{j}
      \ar[d]_-{b} &
      {E^n}
      \ar[u]^-{a \circ u=\ol{a}\circ q}
      \ar[r]^-{0}
      \ar[d]_-{0} &
      {0}
      \ar[u]_-{0}
      \ar[r]^-{0}
      \ar[d]^-{0} &
      {\dots}
      \\
      {\dots} \ar[r]^-{0} &
      {0} \ar[r]^-{0} &
      {Y^{n-1}} \ar[r]^-{0} &
      {0} \ar[r]^-{0} &
      {0} \ar[r]^-{0} &
      {\dots.}
    }
  \end{equation*}
  Consider the morphism of short exact sequences
  \begin{equation}
    \label{eq:extension}
    \xymatrix{
      {0} \ar[r]
      & {M^{n-1}} \ar[r]^-{j} \ar[d]_-{b}
      & {E^n} \ar[r]^-{q} \ar[d]_-{v}
      & {\HH^n(M)} \ar[r] \ar[d]_-{\id}
      & {0} 
      \\ 
      {0} \ar[r]
      & {Y^{n-1}} \ar[r]^-{k'=k'^n}
      & {F^n} \ar[r]^-{r'=r'^n}
      & {\HH^n(M)} \ar[r]
      & {0} 
    }
  \end{equation}
  where the left square is cocartesian. The lower row is a Yoneda
  extension representing
  $\Sigma^n(\pi_{n-1} \circ f \circ \iota_n)$ if we identify
  $\HH^n(M) \xsira{\ol{a}} X^n$.

  The morphism $(b \circ u, b)$ between the lower two rows in
  \eqref{eq:roof} factors as the composition of the following two
  morphisms of complexes (where the first morphism is a
  quasi-isomorphism).
  \begin{equation}
    \label{eq:EM-to-Y-simplified}
    \xymatrix{
      {\dots}
      \ar[r]^-{
        \left(
          \begin{smallmatrix}
            0 & j\\ 0 & 0
          \end{smallmatrix}
        \right)
      }
      &
      {E^{n-1} \oplus M^{n-1}}
      \ar[r]^-{
        \left(
          \begin{smallmatrix}
            0 & j\\ 0 & 0
          \end{smallmatrix}
        \right)
      }
      \ar[d]_-{
        \left(
          \begin{smallmatrix}
            v & 0\\ 0 & b
          \end{smallmatrix}
        \right)
      }
      &
      {E^n \oplus M^n}
      \ar[r]^-{
        \left(
          \begin{smallmatrix}
            0 & j\\ 0 & 0
          \end{smallmatrix}
        \right)
      }
      \ar[d]_-{
        \left(
          \begin{smallmatrix}
            v & 0\\ 0 & b
          \end{smallmatrix}
        \right)
      }
      &
      {E^{n+1} \oplus M^{n+1}}
      \ar[r]^-{
        \left(
          \begin{smallmatrix}
            0 & j\\ 0 & 0
          \end{smallmatrix}
        \right)
      }
      \ar[d]_-{
        \left(
          \begin{smallmatrix}
            v & 0\\ 0 & b
          \end{smallmatrix}
        \right)
      }
      &
      {\dots}
      \\
      {\dots}
      \ar[r]^-{
        \left(
          \begin{smallmatrix}
            0 & k'\\ 0 & 0
          \end{smallmatrix}
        \right)
      }
      &
      {F^{n-1} \oplus Y^{n-1}}
      \ar[r]^-{
        \left(
          \begin{smallmatrix}
            0 & k'\\ 0 & 0
          \end{smallmatrix}
        \right)
      }
      \ar[d]_-{(\ol{b} \circ r',\id)} &
      {F^n \oplus Y^n}
      \ar[r]^-{
        \left(
          \begin{smallmatrix}
            0 & k'\\ 0 & 0
          \end{smallmatrix}
        \right)
      }
      \ar[d]_-{(\ol{b} \circ r',\id)} &
      {F^{n+1} \oplus Y^{n+1}}
      \ar[r]^-{
        \left(
          \begin{smallmatrix}
            0 & k'\\ 0 & 0
          \end{smallmatrix}
        \right)
      }
      \ar[d]_-{(\ol{b} \circ r',\id)} &
      {\dots}
      \\
      {\dots} \ar[r]^-{0} &
      {Y^{n-1}} \ar[r]^-{0} &
      {Y^n} \ar[r]^-{0} &
      {Y^{n+1}} \ar[r]^-{0} &
      {\dots}
    }
  \end{equation}
  
  In order to show injectivity of \eqref{eq:Hom-D-hereditary}
  assume now that $f$ is mapped to zero.  First
  $0=\Sigma^n(\pi_n \circ f \circ \iota_n)=\ol{b} \circ
  \ol{a}^{-1}$ yields $\ol{b}=0$.  Hence $\ol{b} \circ r=0$,
  i.\,e.\ the lower vertical morphism $(\ol{b} \circ r', \id)$ in
  \eqref{eq:EM-to-Y-simplified} equals $(0, \id)$.
  
  Second $\Sigma^n(\pi_{n-1} \circ f \circ \iota_n)=0$ means that
  the extension in the lower row of \eqref{eq:extension} is
  trivial. Hence we can assume that $F^n=Y^{n-1} \oplus \HH^n(M)$
  and that $k'= \left(
    \begin{smallmatrix}
      \id \\ 0
    \end{smallmatrix}
  \right)$ and $r'= (0,\id)$.  Using this, the morphism
  $(\ol{b} \circ r', \id) =(0, \id)$ in
  \eqref{eq:EM-to-Y-simplified} has the form
  \begin{equation*}
    \xymatrix{
      {\dots}
      \ar[r]^-{
        \left(
          \begin{smallmatrix}
            0 & 0 & \id \\ 0 & 0 & 0\\ 0 & 0 & 0
          \end{smallmatrix}
        \right)
      }
      &
      {Y^{n-2} \oplus \HH^{n-1}(M)\oplus Y^{n-1}}
      \ar[r]^-{
        \left(
          \begin{smallmatrix}
            0 & 0 & \id \\ 0 & 0 & 0\\ 0 & 0 & 0
          \end{smallmatrix}
        \right)
      }
      \ar[d]_-{(0,0,\id)} &
      {Y^{n-1} \oplus \HH^n(M) \oplus Y^n}
      \ar[r]^-{
        \left(
          \begin{smallmatrix}
            0 & 0 & \id \\ 0 & 0 & 0\\ 0 & 0 & 0
          \end{smallmatrix}
        \right)
      }
      \ar[d]_-{(0,0,\id)} &
      {\dots}
      \\
      {\dots} \ar[r]^-{0} &
      {Y^{n-1}} \ar[r]^-{0} &
      {Y^n} \ar[r]^-{0} &
      {\dots.}
    }
  \end{equation*}
  But this morphism $(0,0,\id)$ is zero in the homotopy category
  $\K(\mathcal{A})$ (take the homotopy $(\id, 0,0)$).  This shows
  that the morphism $(\ol{b} \circ r', \id)$ vanishes in
  $\K(\mathcal{A})$. Hence $(b \circ u,b)$ vanishes in
  $\K(\mathcal{A})$ and $f$ vanishes in $\D(\mathcal{A})$.  This
  shows that the map \eqref{eq:Hom-D-hereditary} is
  bijective. From this the remaining claims follow immediately.
\end{proof}

\begin{corollary}
  \label{c:equiv-hereditary}
  Let $\mathcal{A}$ be a hereditary abelian category and let
  $\mathcal{U}\subset\mathcal{A}$ be a non-empty, full
  subcategory closed under kernels, cokernels, and extensions.
  Then the obvious functor is an equivalence
  $\D(\mathcal{U}) \sira \D_\mathcal{U}(\mathcal{A})$ of
  triangulated categories.
\end{corollary}

\begin{proof}
  This follows from Theorem~\ref{t:D-hereditary}. First, the
  essential image of the functor
  $\D(\mathcal{U}) \ra \D(\mathcal{A})$ is
  $\D_\mathcal{U}(\mathcal{A})$.  Second, fully faithfulness can
  be tested on objects $X$, $Y$ of $\D(\mathcal{U})$ with
  vanishing differentials, and for those the map
  $\Hom_{\D(\mathcal{U})}(X,Y) \ra \Hom_{\D(\mathcal{A})}(X,Y)$
  is identified with the obvious map
  \begin{equation*}
    \prod_{n \in \bZ} \Hom_\mathcal{U}(X^n, Y^n)
    \times
    \prod_{n \in \bZ} \Ext^1_\mathcal{U}(X^n, Y^{n-1})
    \ra
    \prod_{n \in \bZ} \Hom_\mathcal{A}(X^n, Y^n)
    \times
    \prod_{n \in \bZ} \Ext^1_\mathcal{A}(X^n, Y^{n-1})
  \end{equation*}
  which is an isomorphism by our assumption that $\mathcal{U}$ is
  a full subcategory of $\mathcal{A}$ that is closed under
  extensions.
\end{proof}


\end{document}